\documentclass[11pt,a4paper]{article}
\usepackage{amsmath,amssymb,latexsym,amsthm}
\usepackage[english]{babel}
\usepackage[latin2]{inputenc}

\newtheorem{Theo}{Theorem}[section]
\newtheorem{Lemma}[Theo]{Lemma}
\newtheorem{Prop}[Theo]{Proposition}
\newtheorem{Coro}[Theo]{Corollary}
\theoremstyle{definition}
\newtheorem{defn}[Theo]{Definition}
\newtheorem{Rem}[Theo]{Remark} 
\numberwithin{equation}{section}

\begin{document}

\title{Extensions of groups by weighted Steiner loops}
\author{\'Agota Figula and Karl Strambach} 
\date{}
\maketitle

\begin{abstract}
Solving functional equations given in \cite{nagy} for extensions of a group $A$ by a weighted Steiner loop $S$ we obtain concrete description for all loops with interesting weak associativity properties if the Steiner loop $S$ induces only the trivial automorphism on $A$. We show that the restricted Fischer groups and their geometry play an important role for loop extension with right alternative property. Also the automorphism groups of these extensions as well as the conditions for isomorphisms between two extensions are studied. 
\end{abstract} 

\medskip
\noindent 
Unserem lieben Freunde Heinrich Wefelscheid zum 70. Geburtstag gewidmet. 



\noindent
{\footnotesize {{\em Thanks: } The work of the first author was supported by the Hungarian Scientific Research Fund (OTKA) Grant PD 77392 and by the EEA and Norway Grants (Zolt\'an Magyary Higher Education Public Foundation).} }

\section{Introduction} 
The extension theory in the category of groups is a classical part of group theory which by means of 
cohomology archieved a satisfactory form of high aesthetic value (cf. \cite{huppert}, Kap. I, § 14 and § 16, \cite{kurosh}, Chap. XII, \cite{maclane}).  

In \cite{nagy} the authors investigate thoroughly a variation of extensions which yields loops as extensions of groups by loops such that these extensions are the most natural generalization of Schreier's extension theory for groups. 
These extensions of groups $A$ by loops $S$ are given by two equations describing the action of $S$ on $A$. 

Special types of extensions of groups of order $2$ by Steiner loops play an important role in code theory (cf. \cite{griess}, \cite{hsu}). In \cite{strambach} the authors study extensions of groups of order $2$ by Steiner loops such that the corresponding Steiner triple system has an orientation. 

Calling a Steiner loop $S$ equipped with a function $h: S \setminus \{ e \} \to A$, where $e$ is the identity element of $S$ and $A$ is a group, a weighted Steiner loop we study in this paper Schreier extensions $L$ of groups $A$ by weighted Steiner loops $(S, h)$ such that for the factor system 
$f: S \times S \to A$ one has $f(x,y)= h(x) h(y)$ for all $x,y \in S \setminus \{ e \}$  and $S$ induces only the trivial automorphism on $A$. We call such extensions $L$ Steiner-like loops. 
In general, these loops are power-associative but do not satisfy any other associativity condition. To obtain Steiner-like loops which have some stronger associativity properties we have to decide in which groups the functional equations described in \cite{nagy} have solutions and determine them explicitly. 
In this way we can describe all Steiner-like loops with interesting weak associative properties. To construct such Steiner-like loops one needs in the most cases that the group 
generated by the set $\{ f(x,y); \ x,y \in S \}$ is contained in the centre of $A$. Only the cross inverse or the automorphic inverse property of $L$ force the group $A$ to be abelian. In contrast to this Steiner-like loops satisfying the right Bol identity can be also constructed choosing $A$ a finite simple group having elements of order four (p. 18). 

For all loops considered in this paper the group $D$ generated by the set $\{ h(x); x \in S \setminus \{ e \} \}$ is abelian 
except the loops satisfying only the right alternative or the right inverse property. In this case if the Steiner loop $S$ is finite, then the factor group $D V/V$, where $V$ is a cyclic subgroup of $A$, is a restricted Fischer group.  
In contrast to this, if the group $D$ is abelian,  the Steiner loop $S$ has more than four elements and the values  $f(x,x)$ are determined by the values of the non-constant function $h$, then $D$ has rank two and one generator is an involution. 
 Moreover, $S$ is the direct product of $\mathbb Z_2$ and a subloop 
of index $2$.

The group $G_r$ generated by all right translations of a Steiner-like loop $L$ is an extension of the group $A$ by the group $\Sigma $ generated by the set $\{\rho _{(a,1)};\ a \in S \}$ (Proposition \ref{lefttranslation}). The analogous assertion holds for the group $G_l$ generated by all left translations of a Steiner-like loop $L$ if and only if the set 
$\{ f(x,y); \ x,y \in S \}$ is contained in the centre of $A$ (Proposition \ref{righttranslation}).

Also the automorphism groups of Steiner-like loops as well as the conditions for the isomorphisms between two Steiner-like loops are studied.

\section{Preliminaries} 
A set $L$ with a binary operation $(x,y) \mapsto x \cdot y$ is called a loop 
if there exists an element $e \in L$ such that $x=e \cdot x=x \cdot e$ holds 
for all $x \in L$ and the equations $a \cdot y=b$ and $x \cdot a=b$ for any given $a,b \in L$ have 
precisely one solution which we denote by $y=a \backslash b$ and $x=b/a$.

The subloop $L'$ generated by all associators and all commutators of a loop $L$ is called the derived subloop  of $L$.  
As for groups the loop $L'$ is the smallest subloop of $L$ such that $L/L'$ is an abelian group.  

The left, right and middle nucleus, respectively, of a loop $L$ are the subgroups of $L$ which are defined in the following way: 
\[ N_l=\{ u; \  ux \cdot y= u \cdot xy, \ x,y \in L\}, \ \ N_r=\{ w; \  x \cdot y w= xy \cdot w, \ x,y \in L\} \]
\[ \hbox{and} \ \ N_m=\{ v; \ xv \cdot y= x \cdot vy, \ x,y \in L\}.  \]  
The intersection $N=N_l \cap N_r \cap N_m$ is called the nucleus of $L$. The centre $Z$ of a loop $L$ is the largest subgroup of $N$ such that 
$z x=x z$ for all $x \in L$, $z \in Z$.

The left 
translations $\lambda _a: y \mapsto a \cdot y :L \to L$  as well as the right translations 
$\rho _a: y \mapsto y \cdot a :L \to L$ are bijections of $L$ for any $a \in L$. 

A loop $L$ is flexible if 
\begin{equation} \label{equc1} (x \cdot y) \cdot x = x \cdot (y \cdot x) \end{equation} 
holds for all $x,y \in L$. A loop $L$ is left alternative, respectively right alternative if 
\begin{equation} \label{equc2}  x \cdot (x \cdot y) = (x \cdot x) \cdot y, \end{equation} 
respectively 
\begin{equation} \label{equc4} (y \cdot x) \cdot x = y \cdot (x \cdot x) \end{equation} 
holds for all $x, y \in L$. 
A loop $L$ satisfies the left inverse, respectively the right inverse, respectively the cross inverse property 
if for all $x, y \in L$   
\begin{equation}  \label{equc3} x^{\lambda } \cdot (x \cdot y)= y, \end{equation} 
respectively 
\begin{equation} \label{equc6} (y \cdot x) \cdot x^{\rho }= y, \end{equation} 
respectively 
\begin{equation} \label{equk1} (x \cdot y) \cdot x^{\rho }= y \end{equation} 
holds with $x^{\lambda }=e/x$, $x^{\rho }=x \backslash e$.  
If a loop $L$ satisfies the left or the right inverse property, then one has 
$x^{\lambda }=x^{\rho }$ and we denote this element $x^{-1}$. A loop $L$ has the automorphic inverse property if the identity  
\begin{equation} \label{equk2} (x \cdot y)^{-1}= x^{-1} \cdot  y^{-1} \end{equation} 
holds for all $x, y \in L$. 
A loop $L$ satisfies the weak inverse property if for  $x,y,z \in L$ one has 
\begin{equation} \label{equweek} x \cdot (y \cdot z)=e \ \ \hbox{whenever} \ \  (x \cdot y) \cdot z =e.  \end{equation} 
A loop $L$ is a left, respectively a right Bol loop if 
\begin{equation} \label{equk3} [x \cdot (y \cdot x)] \cdot z = x \cdot [y \cdot (x \cdot z)],  \end{equation} 
respectively 
\begin{equation} \label{equk4} z \cdot [(x \cdot y) \cdot x] = [(z \cdot x) \cdot y] \cdot x  \end{equation} 
holds for all $x, y, z \in L$. 
A loop $L$ which has the left and the right Bol identity is called a Moufang loop.

\noindent
A Steiner triple system $({\mathfrak S}, T)$ is an incidence structure consisting of a set ${\mathfrak S}$ of points and a set $T$ of blocks such that two distinct points are contained in precisely one block and every block has precisely three points. 
A finite Steiner triple system with $n$ points exists if and only if $n \equiv  1$ or $3 \ (\hbox{mod} \ 6)$ 
(cf. \cite{flugfelder}, V.1.9 Definition, p. 124).  

A loop $L$ is called a Steiner loop, if it is totally symmetric, i.e. if $x \cdot y = y \cdot x$ and 
$x \cdot (x \cdot y)=y$ hold for all $x,y \in L$. Putting $y=e$ into the last identity we have $x^2=e$ for all $x \in L$. Hence a 
Steiner loop is a commutative loop of exponent $2$ with the inverse property and the bijections $\lambda _a = \rho _a$, $a \in L$, are involutions (see \cite{flugfelder}, Chap. V.). 

With a Steiner triple system $({\mathfrak S}, T)$ is associated a Steiner loop $(S({\mathfrak S}), \circ )$ such that the elements of 
$S({\mathfrak S}) \setminus \{ e \}$, where $e$ is the identity of $S({\mathfrak S})$, are the points of the Steiner triple system $({\mathfrak S}, T)$, 
the product $a \circ b$ is the third point of the line determined by $a, b$ and  $a \circ a = e$ for all 
$a \in S({\mathfrak S})$.     

\begin{defn}
A weighted Steiner
loop $(S, h)$ is a Steiner loop $S$ and a map $h:S \setminus \{ e \} \to A$, where $A$ is a group.  
\end{defn} 

A restricted Fischer group is a pair $(G, E)$ consisting of a group $G$ and a system of generators $E \subset G$ satisfying the following conditions: 
\newline
\noindent
(i) For all $x \in E$ one has $x^2=1$. 
\newline
\noindent
(ii) For all $x,y \in E$ we have $(xy)^3=1$ and $xyx \in E$ (cf. \cite{manin}, p. 20).

\section{Extensions of loops} 
Let $S$ be a loop and $A$ be a group. In \cite{nagy} are considered loop extensions
\begin{equation} \label{equ1} 1 \longrightarrow A \longrightarrow L \longrightarrow S \longrightarrow  e  \nonumber  \end{equation}
such that the multiplication in the loop $L$, which is realized on the set
$S \times A$, is given by
\begin{equation} \label{equextension} (x, \xi )(y, \eta ) = (xy, f(x, y) \xi \eta ), \end{equation} 
where $f$ is a map from $S \times S$ to $A$ with $f(x, e) =  f(e, y) = 1 \in A$ for all $x,y \in S$.   The identity element of the loop $L$ is $(e,1)$. 

\noindent
The left inverse of $(x, \xi) \in L$ is $(x^{-1}, f(x^{-1},x)^{-1} \xi ^{-1})$ and
the right inverse of $(x, \xi)$ is $(x^{-1}, \xi ^{-1} f(x,x^{-1})^{-1})$ such that $x^{-1}= x \backslash e= e /x$ in $S$. Hence in the loop $L$ the left inverse of any element coincides with the right inverse if and only if 
$f(x^{-1},x)= \xi^{-1} f(x, x^{-1}) \xi $ for all $x \in S$ and $\xi \in A$ (cf. \cite{nagy}, Proposition 3.3, p. 762).

\noindent
The group $A$ is contained  in the middle  and in the right nucleus of $L$. It is contained in the left nucleus and hence in the nucleus of $L$ if and only if $f(x, y)$ are elements of the centre of $A$ for all $x, y \in S$ (cf. \cite{nagy}, Proposition 3.2. (i), p. 762).

\begin{Rem} \label{rem1}
A loop $L$ of type (\ref{equextension}) is a central extension of $A$ by the loop $S$ if and only if $A$ is commutative (cf. \cite{nagy}, p. 762). The loop 
$L$ is the direct product of $A$ and $S$ if and only if $f(x,y)=1$ for all $x,y \in S$. 
\end{Rem}

\noindent
{\bf Construction:} Let $\Delta$ be the group generated by the set $\{f(x, y);\   x, y \in S \}$. If $\Delta $ is central in $A$, then $A$ is a central extension of $\Delta$ by a group $B$ and the multiplication of $A$ on the set $B \times \Delta$ is given by 
$(r_1,s_1) (r_2,s_2)=(r_1 r_2, s_1 s_2 k(r_1,r_2))$,  
where the map $k: B  \times B \to \Delta $ is a factor system (\cite{huppert}, Satz 14.1, p. 86). The group $\Delta$ is also a central subgroup of the loop $L$. Hence $L$ is the direct product of $A$ and the smallest subloop of $L$ containing $\Delta$ and $S$ 
with amalgamated subgroup $\Delta $. The multiplication of $L$ on the set $S \times B \times \Delta$ is given by 
\begin{equation} ((a_1,1)(e, \rho _1)(e, \sigma _1))((a_2,1)(e, \rho _2)(e, \sigma _2))= \nonumber \end{equation} 
\begin{equation} \label{equ2} (a_1 a_2,1) (e, \rho _1 \rho _2) (e, \sigma _1 \sigma _2 f(a_1,a_2) k(\rho _1, \rho _2)), \end{equation}  
where $a_i \in S$, $\rho _i \in B$, $\sigma _i \in \Delta$, $i=1,2$. 

\medskip
\noindent 
If $S$ is a group, then $L$ is also a group precisely if $f: S \times S \to \Delta$ is a factor system for $S$. From this it  follows:

\noindent
\begin{Rem} If the loop $S= \{ e,a \}$ and $f(a,a)$ is contained in the centre of $A$, then the loop $L$ of type (\ref{equextension}) is a group.  
\end{Rem}

\noindent
Since extensions of type (\ref{equextension}) do not allow a classification in this paper we treat a particularly simple type of extensions (\ref{equextension}). 

\begin{defn} 
Let $A$ be a group and let $(S, h)$ be a weighted Steiner loop.  
A Steiner-like loop is an extension (\ref{equextension}) such that for the map $f: S \times S \to A$ one has 
$f(x, y) = h(x)h(y)$ for all $x, y \in S \setminus \{ e \}$ and $x \neq y$.  
\end{defn} 

Since $S$ is a Steiner loop we have $x^{-1}=x$ and $f(x^{-1},x)=f(x,x)$. Hence the left inverse of an element of $L$ coincides with the right inverse if and only if the set $\{f(x,x); \ x \in S \}$ is contained in the centre of $A$. A Steiner-like loop $L$ is commutative if and only if $A$ is commutative.  

A Steiner-like loop is power-associative precisely if 
the set $\{f(x,x); \ x \in S \}$ is contained in the centre of $A$.
In general, if $S$ is a finite Steiner loop with $n$ elements and if we take for the group $A$ an extension of an abelian group $Z$ by a finitely generated group $D$, then mapping $x \mapsto h(x)$, $x \in S \setminus \{ e \}$, where $h(x)$ is a generator of $D$, we obtain a Steiner-like loop $L$, if the number of the generators of $D$ does not exceed $n-1$ and the set $\{ h(x); x \in S \setminus \{ e \} \}$ contains a set of generators of $D$.

\section{Restricted Fischer groups and weighted Steiner triple systems} 
\begin{defn}
Let $({\mathfrak S},T)$ be a finite Steiner triple system and $G$ be a group generated by a set ${\mathcal I}$ of involutions. If $w$ is a mapping from ${\mathfrak S}$ onto ${\mathcal I}$ such that 
\begin{equation} \label{wfunction} w(x) w(y) w(x)= w(xy), \end{equation} 
where $x, y$ and $xy$ are points of a block $t$ of $T$, then we say that $({\mathfrak S}, w, {\mathcal I})$ is a weighted Steiner triple system. 
\end{defn} 
The group generated by $w(x)$ and $w(y)$ for different points $x$ and $y$ is according to (\ref{wfunction}) either the group of order $2$ or the symmetric group of order $6$. If for any two points $x$ and $y$ the group 
$\langle w(x), w(y) \rangle $ is abelian, then $G$ is isomorphic to $\mathbb Z_2$. Let $t= \langle x,y,z \rangle $ be a block of $T$ such that $\langle w(x), w(y) \rangle $ is the symmetric group of order $6$. Let $V_1, \cdots , V_k$ be the maximal sets of points in ${\mathfrak S}$ having the values $w_1, \cdots , w_k$ and one has ${\mathcal I}= \{ w_1, \cdots , w_k \}$. Because of 
(\ref{wfunction}) the sets $V_i$ form Steiner triple subsystems (which can be also points) such that $w_{a_i}=w_i$ for all points $a_i \in V_i$ and $V_i \cap V_j = \emptyset $. Moreover, one has $w_i w_j \neq w_j w_i$ for all $i \neq j$ and hence 
the group $\langle w_i, w_j \rangle $ for all $i \neq j$ is the symmetric group of order $6$. For the elements of 
the set ${\mathcal I}$ one has $o( w_i w_j)=3$ and $w_i w_j w_i \in {\mathcal I}$. Hence $G$ is a restricted Fischer group with ${\mathcal I}$ as a set of generators (\cite{manin}, p. 20). The Fischer space $\mathcal{L(I)}$ corresponding to the Fischer group $G$ has as points the elements of ${\mathcal I}$ and as lines the triples $\{ w_i, w_j, w_i w_j w_i \}$ for all 
$i \neq j \in {\mathcal I}$ (\cite{aschbacher}, p. 92). The number of elements of ${\mathcal I}$ is $3^k$. According to 
\cite{aschbacher}, p. 94, the Fischer space $\mathcal{L(I)}$ is a Hall system, i.e. any three non-collinear points of 
${\mathcal I}$ generates the affine plane of order $3$ and the group $G$ consists of automorphisms of $\mathcal{L(I)}$. 
Let $\varphi : {\mathfrak S} \to \mathcal{L(I)}$ be the mapping $x \mapsto w(x)$. Then $\varphi $ is a homomorphism of Steiner triple systems since for the points $x, y, xy$ of a block $t$ one has either $w(x)= w(y)= w(xy)$ or $w(x)$, $w(y)$, $w(xy)$ are 
different points of a block in $\mathcal{L(I)}$. 

Conversely, let $G$ be a restricted Fischer group generated by a set $\mathcal{I}$ of involutions. Let $\mathcal{L(I)}$ be the corresponding Fischer space. We call a Steiner triple system ${\mathfrak S}$ a covering of $\mathcal{L(I)}$ if there is an epimorphism 
$\varphi : {\mathfrak S} \to \mathcal{L(I)}$. Then $S$ is a weighted Steiner triple system with values in $\mathcal{I}$ if we put 
$w(x)=r$ for all $x \in \varphi^{-1}(r)$ and $r \in \mathcal{L(I)}$. Summarizing the above discussion we obtain 

\begin{Theo} \label{covers} A Steiner triple system ${\mathfrak S}$ is a weighted Steiner triple system such that the group $G$ generated by the set $\mathcal{I}$ of involutions $w(x)$, $x \in {\mathfrak S}$, is non-abelian if and only if $G$ is a restricted Fischer group and ${\mathfrak S}$ covers the Fischer space $\mathcal{L(I)}$.
\end{Theo}      

M. Hall is classified all restricted Fischer groups of rank $4$ determinating the largest restricted Fischer group of rank $4$. This group has order $2 \cdot 3^{10}$ and its homomorphic images of rank $4$ give all restricted Fischer groups of rank $4$. By this he obtained also a classification of all Hall systems which can be generated by four distinct points (see 19, in \cite{aschbacher}, pp. 94-101). 

The simplest cases of coverings of Fischer space arise from epimorphisms of affine spaces over the field $GF(3)$. Let $G$ be the semidirect product of the elementary abelian $3$-group $(V, + )$ of order $3^s$ by a group $\langle \alpha \rangle $ 
of order $2$ acting on $V$ as $x \mapsto -x$ for all $x \in (V, + )$. Then $G$ is a restricted Fischer group such that the Fischer space $\mathcal{L(I)}$ corresponding to the set $\mathcal{I}= \{(x, \alpha ); x \in V \}$ of generating involutions of $G$ is the affine space over $GF(3)$ of dimension $s$. The mapping 
$(x_1, \cdots ,x_s) \mapsto ((x_1, \alpha ), \cdots , (x_s, \alpha ))$ is an isomorphism from $(V, + )$ onto $\mathcal{L(I)}$. Let $W$ be the affine space over $GF(3)$ of dimension $n$ with $n > s$. Then the mapping $\tau : (z_1, \cdots , z_n) \mapsto (z_1, \cdots , z_s, 0, \cdots ,0)$ is an epimorphism from $W$ onto $V$. One has $\tau ^{-1}(z_1, \cdots ,z_s)=\{ (z_1, \cdots , z_s,y_{s+1}, \cdots , y_n); y_l \in GF(3), s+1 \le l \le n \}$. Assigning to any element $u$ of 
$\tau ^{-1}(z_1, \cdots ,z_s)$ the value $w(u)=((z_1, \alpha ), \cdots , (z_s, \alpha ))$ we see that $W$ is a covering of the Fischer space $\mathcal{L(I)}$.

\begin{Prop} Let $({\mathfrak S}, w, {\mathcal I})$ be a weighted Steiner triple system such that the mapping 
$w: {\mathfrak S} \to {\mathcal I}$, 
$x \mapsto w_x$ is bijective and the group $G$ generated by the set ${\mathcal I}$ is non-abelian, then the multiplication $\ast $ on ${\mathfrak S}$ defined by $x \ast y= x \cdot y$ for all $x \neq y \in {\mathfrak S}$ and $x \ast x=x$ gives a distributive symmetric quasigroup. 
\end{Prop} 
\begin{proof} By Theorem \ref{covers} the group $G$ is a restricted Fischer group and the set 
${\mathcal I}=\{ w_x ; x \in {\mathfrak S} \}$ is a set of generators of $G$. Hence one has 
$w_x w_y w_x = w_{xy}$. Since the mapping $w : x \mapsto w_x$ 
is bijective we have $w^{-1}(w_{xy})= x \cdot y= x \ast y$. The set ${\mathcal I}$ is with respect to the multiplication 
$(w_x \circ w_y) = w_x w_y w_x$ a distributive symmetric quasigroup 
(cf. \cite{bolis}, p. 387). Since 
$w(x \ast y) =  w_{xy} = w_x w_y w_x$ we have 
\[ (w_x \circ (w_y \circ w_z))= w_x \circ (w_y w_z w_y )=  w_x \circ w(y \ast z) = \] 
\[ w_x  w(y \ast z) w_x = w( x \ast (y \ast z)) \] 
and 
\[((w_x \circ w_y) \circ (w_x \circ w_z)) = w(x \ast y) \circ w(x \ast z) = \]
\[ w(x \ast y) w(x \ast z) w(x \ast y) = w((x \ast y) \ast (x \ast z)). \]  
As $(w_x \circ (w_y \circ w_z)) = ((w_x \circ w_y) \circ (w_x \circ w_z))$, then using $w^{-1}$ we get that 
$({\mathfrak S}, \ast )$ is a distributive symmetric quasigroup. \end{proof}

\section{Weights of Steiner loops $S$}

\begin{Lemma} \label{hxcommutative} Let $L$ be a Steiner-like loop  
such that the  Steiner loop $S$ has more than $2$ elements and the set 
${\mathcal K}=\{ h(x) h(y); \ x,y \in S \setminus \{ e \}, x \neq y \}$
is contained in the centre $Z(A)$ of the group $A$. Then the group $D$ generated by 
the set ${\mathcal H}=\{ h(z); \ z \in S \setminus \{ e \} \}$ is abelian. 
\end{Lemma}
\begin{proof} As $h(x) h(y)=z_1 \in Z(A)$ and $h(y) h(x)=z_2 \in Z(A)$ we have 
$h(x)= z_1 h(y)^{-1}$ and hence $z_1=z_2$. It follows that the 
range of the function $h$ is commutative and hence the group $D$ is abelian. 
\end{proof}

\begin{Lemma} \label{hxinvolution} Let $L$ be a Steiner-like loop. Let $D$ be the group generated by the set 
${\mathcal H}=\{ h(z); \ z \in S \setminus \{ e \} \}$. We assume that the set  ${\mathcal F}=\{f(x,x); \ x \in S \}$  is contained in the centre $Z(A)$ of $A$. 
If for all $x, y \in S \setminus \{ e \}$ with $x \neq y$ the identity 
\begin{equation} \label{smallequ3uj} h(x) h(y) h(x) h(x y) = f(x,x)  \nonumber \end{equation} 
holds, then the following properties are satisfied: 
\newline
\noindent
a) The set $\{ h(x)^2; \ x \in S \setminus \{ e \} \}$ is contained in the centre $Z(D)$ of $D$. 
\newline
\noindent
b) For arbitrary block $t=\{ x, y, xy \}$ of the Steiner triple system $(S \setminus \{ e \}, T)$ belonging to $L$ let 
$H_t$ be the group generated by the set  
$\{ h(x), h(y), h(xy) \}$. Then the following holds: 
\newline
\noindent
(i) If the factor group $H_t/(H_t \cap Z(D))$ is different from the identity, then it is isomorphic either to $S_3$ or to $\mathbb Z_2$. If the factor group $H_t/(H_t \cap Z(D))$ is the identity or isomorphic to $\mathbb Z_2$, then the group 
$H_t$ is abelian. 
\newline
\noindent
(ii) The group $D$ is abelian if and only if  for every block $t \in T$ the group $H_t$ is abelian. 
\end{Lemma} 
\begin{proof}  As $h(xy) h(x) h(xy) h(y) = f(xy, xy)= h(xy) h(y) h(xy) h(x)$ one has  
\begin{equation} \label{equinvolution1} h(x) h(xy) h(y)= h(y) h(xy) h(x). \end{equation} 
Since $f(x,x) \in Z(A)$ for all $x \in S$ we have  
\begin{equation} \label{equinvolution2} h(x)^{-1}[ h(x) h(y) h(x) h(xy)] h(x)= h(x) h(y) h(x) h(x y) = h(y) h(x) h(xy) h(x)  \end{equation}  
and  
\begin{equation} \label{equinvolution3} h(y) [h(x) h(y) h(x) h(xy)]= [h(x) h(y) h(x) h(x y)] h(y). \end{equation}   
Using (\ref{equinvolution2}) on the left hand side of (\ref{equinvolution3}) and (\ref{equinvolution1}) on the right hand side of 
(\ref{equinvolution3}) we obtain 
$h(y)^2 h(x) h(xy) h(x)= h(x) h(y)^2 h(xy) h(x)$ or equivalently $h(y)^2 h(x) = h(x) h(y)^2$ for all 
$x \neq y \in S \setminus \{ e \}$. Hence the assertion a) is proved. 

For elements $h(x) Z(D)$ of $D/Z(D)$ one has $h(x) h(y) h(x) h(xy) Z(D)= Z(D)$. As 
\begin{equation} h(xy) h(x) Z(D)= [h(x) h(y) h(x) h(xy)] h(xy) h(x) Z(D)= \nonumber \end{equation}
\begin{equation} h(x) h(y) h(xy)^2 h(x)^2 Z(D)  \nonumber \end{equation}  
we get  $h(xy) h(x) Z(D)= h(x) h(y) Z(D)$.  Analogously we have 
\begin{equation} \label{equinvolution4}  
h(x) h(y) Z(D)= h(xy) h(x) Z(D)= h(y) h(xy) Z(D), \nonumber \end{equation}
\begin{equation} h(y) h(x) Z(D)= h(x) h(xy) Z(D)= h(xy) h(y) Z(D). \end{equation} 
Let $t$ be an arbitrary block of $(S \setminus \{ e \}, T)$ consisting of points $x,y, xy$. The factor group $H_t/Z_t$, where $Z_t=H_t \cap Z(D)$ is generated by the elements $h(x) Z_t$, $h(y) Z_t$, $h(xy) Z_t$.  If $h(x) Z(D)= h(y) Z(D)= h(xy) Z(D)= Z(D)$, then $H_t/Z_t$ is the identity and the group $H_t$ is abelian. 
If $h(x) Z(D) \neq Z(D)$, then it follows from a) that $h(x) Z(D)$ is an involution.  
If all generators are different involutions, then because of (\ref{equinvolution4}) we have 
\begin{equation} \label{involutionujequ4} 
[h(x) h(y)]^3 Z_t= h(x) [h(y) h(x)] [h(y) h(x)] h(y) Z_t= \nonumber \end{equation} 
\begin{equation} 
h(x) h(x) h(xy) h(xy) h(y) h(y) Z_t=Z_t. \nonumber \end{equation} 
Hence $h(x) h(y) Z_t$ generates a group $\Psi $ of order $3$ and $h(y) h(x) Z_t$ is the inverse of $h(x) h(y) Z_t$. 
Moreover, it follows from relations (\ref{equinvolution4}) that all three elements $h(x) Z_t$, $h(y) Z_t$, $h(xy) Z_t$ invert any element of $\Psi $. 
Hence $H_t/Z_t$ is isomorphic to $S_3$.   
If two generators of $H_t/Z_t$ coincide, then $H_t/Z_t$  is isomorphic to the cyclic group of order $ \le 2$. As 
$H_t \cap Z(D) \le Z(H_t)$ in this case the group $H_t=\langle h(x), h(y), h(xy) \rangle $ is abelian and the assertion (i) is proved. 

If for every block $t \in T$ the group $H_t$ is abelian,  
then one has $h(x) h(y)= h(y) h(x)$ for all $x \neq y \in S \setminus \{ e \}$. Hence the range of $h$ is commutative which yields that the group $D$ is abelian and the assertion (ii) follows.    \end{proof}

\begin{Theo}  \label{Prop3} Let $L$ be a Steiner-like loop. We assume that the set ${\mathcal F}=\{ f(x,x); x \in S \}$ is contained in the centre $Z(A)$ 
of the group $A$. Let $D$ be the group generated by the set 
${\mathcal H}=\{ h(z); \ z \in S \setminus \{ e \} \}$. If $S$ has more than $4$ elements, then the following properties are equivalent: 
\newline
\noindent
a) For all $x, y \in S \setminus \{ e \}$ with $x \neq y$ the identity 
\begin{equation} \label{smallequ3} h(x) h(y) h(x) h(x y) = f(x,x)  \end{equation} 
holds. 
\newline
\noindent
b) If the group $D$ is abelian, then precisely one of the following cases holds: 
\newline
\noindent
(i) For all $x \in S \setminus \{ e \}$ one has $h(x)=t$, $f(x,x)=t^4$ and  $D=\langle t \rangle $. 
\newline
\noindent
(ii) There is a subloop $U$ of the Steiner loop $S$ such that $S$ is the direct product of $U$ and $\mathbb Z_2$ and 
$h(x)=t \neq 1$ 
for all $x \in U$ whereas $h(y)=t \omega = \omega t$ with a fixed involution $\omega $ for all $y \in S \setminus U$. The elements $f(x,x)$ have the form $t^4$ or $t^4 \omega $ depending whether $x \in U$, respectively $x \notin U$. The group $D$ is the direct product of $\langle t \rangle $ and $\langle \omega \rangle $.  
\newline
\noindent
If the group $D$ is non-abelian and the Steiner loop $S$ is finite, then one has $h(x)=u \omega _x = \omega _x u$, where $u$  is  a  
fixed \ element \ of \ $A$ \ which \ centralizes \ any \ element \ of \ $D$ \ and \ for \ the \ order \ $o(\omega _x)$ \ of \ the \ element \ $\omega _x$ \ of \ $A$ \ one \ has \  $o(\omega _x) \le 2$. For every $x \in S \setminus \{ e \}$ one has  $f(x,x)=u^4$.   If $\Gamma $ is the set of different involutions in the factor group $D \langle u \rangle /\langle u \rangle = \langle \omega _x \langle u \rangle , x \in S \setminus \{ e \}  \rangle $, then $D \langle u \rangle /\langle u \rangle $ is a restricted Fischer group generated by $\Gamma $.

\medskip 
\noindent 
The Steiner-like loop $L$ is the direct product of $A$ and $S$ precisely if $D$ has order $\le 2$.  
\end{Theo} 
\begin{proof}  First we assume that identity (\ref{smallequ3}) is satisfied.
Then we have   
\begin{equation} \label{equ4} h(y) h(x) h(y) h(xy)= f(y,y) \end{equation} 
and
\begin{equation} \label{equ5} h(xy) h(y) h(xy) h(x)= f(xy,xy). \end{equation} 
Moreover, from (\ref{smallequ3}) it follows
\begin{equation} \label{equ6} h(xy)= h(x)^{-1} h(y)^{-1} h(x)^{-1} f(x,x). \end{equation} 
Putting (\ref{equ6}) into equation (\ref{equ4}), respectively into equation (\ref{equ5}) we obtain 
\begin{equation} \label{equ7}  h(x)^{-1} h(y)^{-1} h(x)^{-1} f(x,x)= h(y)^{-1} h(x)^{-1} h(y)^{-1} f(y,y) \end{equation}
respectively  
\begin{equation} \label{equ8} h(x)^{-1} h(y)^{-1} h(x)^{-1} f(x,x) h(y) h(x)^{-1} h(y)^{-1} f(x,x)= f(xy,xy).  \end{equation} 
Putting  (\ref{equ7}) into (\ref{equ8}) we have  
\begin{equation} \label{equ9} h(y)^{-1} h(x)^{-2} h(y)^{-1} f(x,x) f(y,y)= f(xy,xy). \end{equation}
Since $f(x,x) f(y,y) f(xy,xy)^{-1} \in Z(A)$ and $\{h(x)^2; \ x \in S \setminus \{ e \} \} \in Z(D)$ (see Lemma \ref{hxinvolution} a))  
 equation (\ref{equ9}) yields 
\begin{equation} \label{equ10} f(x,x) f(y,y) f(xy,xy)^{-1}= h(x)^2 h(y)^2= h(y)^2 h(x)^2.  \end{equation} 
From (\ref{equ10}) it follows  
\begin{equation} \label{equ14} h(xy)^2 h(y)^2=f(xy, xy) f(y,y) f(x,x)^{-1}, \end{equation} 
\begin{equation} \label{equ15} h(x)^2 h(xy )^2= f(x,x) f(xy, xy) f(y,y)^{-1}, \end{equation} 
\begin{equation} \label{equ16} f(xy, xy)=f(x,x) f(y,y) h(x)^{-2} h(y)^{-2}.  \end{equation}
Putting (\ref{equ16}) into equations (\ref{equ14}) and (\ref{equ15}) we obtain 
\begin{equation} \label{equ17} h(xy)^2 h(y)^2= f(y,y)^2 h(x)^{-2} h(y)^{-2} \end{equation} 
and  
\begin{equation} \label{equ18} h(xy )^2= f(x,x)^{2} h(x)^{-4} h(y)^{-2}. \end{equation} 
The substitution of (\ref{equ18}) in  equation (\ref{equ17}) yields 
\begin{equation} \label{equ19}  f(x,x)^2 h(x)^{-2}= f(y,y)^2 h(y)^{-2}. \end{equation} 
Hence $f(x,x)^2 = r h(x)^2$ for all $x \in S \setminus \{ e \}$, where $r$ is a fixed element of $A$.   
Squaring equation (\ref{equ10}) and putting there $f(x,x)^2 = r h(x)^2$ we have 
\begin{equation} \label{equ21} h(x)^2 h(y)^2 h(x y)^2=r \ \hbox{or \ equivalently} \ h(y)^2 h(xy)^2=r h(x)^{-2} \end{equation} 
for all $x \neq y \in S \setminus \{ e \}$.  
Leaving $x$ fixed we obtain for $y_1 \neq y \in S \setminus \{ e \}$  
\begin{equation} \label{equ22} h(y_1)^2 h(x y_1)^2= r h(x)^{-2}. \end{equation} 
From (\ref{equ21}) and (\ref{equ22}) it follows 
\begin{equation} \label{equ23} h(y)^2 h(y_1)^{-2} = h(x y_1)^2 h(x y)^{-2}. \end{equation} 
Putting  $z \neq x \in S \setminus \{ e \}$ into (\ref{equ23}) we obtain 
\begin{equation} \label{equ24} h(y)^2 h(y_1)^{-2} = h(z y_1)^2 h(z y)^{-2}. \end{equation}  
Comparing  equations (\ref{equ23}) and (\ref{equ24}) one has 
\begin{equation} \label{equ25} h(y)^2 h(y_1)^{-2} = h(x y_1)^2 h(x y)^{-2}= h(z y_1)^2 h(z y)^{-2}=s \end{equation} 
for all $x \neq z \neq y \neq y_1 \in S \setminus \{ e \}$. Interchanging the variables 
$y \neq y_1 \in S \setminus \{ e \}$ we obtain  $h(y_1)^2 h(y)^{-2}=s= h(y)^2 h(y_1)^{-2}$. This yields $s^2=1$ and $h(y)^4 = h(y_1)^4$ for all $y, y_1 \in S \setminus \{ e \}$. Since $h(yy_1)^2 = h(y_1)^2= s h(y)^2$ for all 
$y_1 \neq y \in S \setminus \{ e \}$ using equation (\ref{equ21}) for $y, y_1, y y_1$ we get $h(y)^2 s^2 h(y)^4 =r$ or equivalently $h(y)^6=r$ for all $y \neq S \setminus \{ e \}$. Since for all $x, y \in S \setminus \{ e \}$ one has $h(x)^6= h(y)^6$ and $h(x)^4 =h(y)^4$ we obtain  $h(x)^2=h(y)^2=w$.  
Hence $h(x)= u \omega _x= \omega _x u$ with $u^2=w$ and the order of the element $\omega _x$ is $\le 2$. 
Since $h(x)=u \omega _x$ identity (\ref{smallequ3}) gives that 
$f(x,x)=u^4 \omega _x \omega _y \omega _x \omega _{xy}$.   

If the function $h: S \setminus \{ e \} \to A$ is constant, then $h(x)=t$ for all $x \in S \setminus \{ e \}$ and 
$D= \langle t \rangle $. This gives case (i). 

If $D$ is abelian, but the function $h$ is not constant, then $D \langle u \rangle $ is abelian and by Lemma \ref{hxinvolution} the set $\Gamma $ of different involutions in $D \langle u \rangle / \langle u \rangle $ consists of the element $\omega \langle u \rangle $ since otherwise for the different involutions $\omega _x, \omega _y$ 
the factor group $H_t/Z_t \cong \mathbb Z_2 \times \mathbb Z_2$, where $H_t =\langle u \omega _x, u \omega _y, u \omega_{xy} \rangle $ for the  line $t$ joining the points $x$ and $y$ of $S \setminus \{ e \}$. As $h(x)= u \omega _x$ with $\omega _x=1$ or $\omega _x=\omega $ and $h$ is not constant we have $u \neq u \omega $ and the group $D$ is the direct product $\langle u \rangle \times \langle \omega \rangle $ of order $\ge 2$.  

Since any two points $x,y $  in the Steiner triple system $S \setminus \{ e \}$ are joined by a line
we have  \[ h(x) h(y) h(xy)= f(x,x)h(x)^{-1} =f(y,y) h(y)^{-1}=h(y) h(x) h(xy)= l \] for suitable constant $l$ in the abelian group $D$. As $h(x)= u$ or $h(x)=u \omega $ for every $x \in S \setminus \{ e \}$ we get $l =u^3$ or $l=u^3 \omega $. 

Let 
$l= u^3$. Since the function $h$ is not constant there is a point $x \in S \setminus \{ e \}$ such that $h(x)=u$. If for all elements $z \in S \setminus \{ e, x \}$ one has $h(z)=u \omega $, then we would have a line $t=\{ z,w,zw \}$ such that 
$h(z) h(w) h(zw)= u^3 \omega $. Now we consider the set $\Delta \subset S \setminus \{ e \}$ such that for every point 
$x \in \Delta $ one has $h(x)=u$. The elements $\Delta \cup \{ e \}$ form a Steiner subloop $U$ of the Steiner loop $S$. For the elements $z \in S \setminus U$ one has $h(z)= u \omega $ and for any two elements  of $S \setminus U$ the third point of the joining line lies in $U$. Identifying the points of $U$ with the pairs $(a, e)$, $a \in U$, and a point $z \in S \setminus U$ with the pair $(e,z)$ we see that $S$ is the direct product of $U$ and $\mathbb Z_2$. 

If $l=u^3 \omega $ proceeding analogously we obtain that $S$ is the direct product of the subloop $U$ and $\mathbb Z_2$ such that $U \setminus \{ e \}$ consists of elements $z$ with $h(z)=u \omega $ whereas for the elements $x$ of the set $S \setminus U$ one has $h(x)=u$. This yields assertion (ii).   

Let $D$ be non-abelian and let $\Gamma $ be the set of all different involutions in the factor group 
$D \langle u \rangle / \langle u \rangle =  \langle \omega _x \langle u \rangle , x \in S \setminus \{ e \}  \rangle $. Since any two involutions in $\Gamma $ generates the symmetric group $S_3$ (cf. Lemma \ref{hxinvolution}) we have $f(x,x)=u^4$ and the group $D \langle u \rangle / \langle u \rangle $ is a restricted Fischer group (cf. Theorem \ref{covers}).  
Hence the assertion b) follows from a).  

Conversely, if $D$ is not abelian, then $\omega _x \omega _y \omega _x \omega _{xy}=1$ and $f(x,x)=h(x) h(y) h(x) h(xy)=u^4$. 
If the group $D$ is abelian and for all $x \in S \setminus \{ e \}$ one has $h(x)=t$, then $f(x,x)=h(x) h(y) h(x) h(xy)=t^4$. Depending whether $x \in U$ or $x \notin U$ we obtain $f(x,x)=h(x)h(y)h(x)h(xy)$ equals to $t^4$, respectively 
$t^4 \omega $.

It follows from Remark 3.1 that the loop $L$ is the direct product of $A$ and $S$ if and only if $\omega _x= \omega _y$ for all $x, y \in S \setminus \{ e \}$. Then the group $D$ is abelian and generated by $t$ and $\omega $. For $x \in U$ and 
$y \in S \setminus U$ we obtain $1=h(x) h(y)= t^2 \omega $, whereas for $x, z \in U$ we get $1=h(x) h(z)=t^2$. Hence one has 
$\omega =1$ and the last assertion of the theorem is proved.    
\end{proof}

\begin{Prop} \label{Propuj3} Let $L$ be a Steiner-like loop such that $S=\{ e,x,y, xy \}$ is the elementary abelian group of order $4$. We assume that the set ${\mathcal F}=\{ f(x,x); x \in S \}$ is contained in the centre $Z(A)$ of the group $A$. Let $D$ be the group generated by the set 
${\mathcal H}=\{ h(z); \ z \in S \setminus \{ e \} \}$. Then properties a) and b) are equivalent: 
\newline
\noindent
a) For all $x, y \in S \setminus \{ e \}$ with $x \neq y$ identity (\ref{smallequ3}) in Proposition (\ref{Prop3}) 
holds. 
\newline
\noindent
b) If the group $D$ is non-abelian, then  
it is a product $K \cdot \langle a \rangle $, where $a=h(x)$ and $t=h(x) h(y)$,  such that the abelian group $K$ has the form 
$\langle s \rangle \times Z(D)$, where the cyclic 
group $\langle s \rangle $ of order $3$ is the commutator subgroup $D'$ of $D$, the group $Z(D)$ is generated by $a^2$ and $t^3$ and $a$ inverts any element of $\langle s \rangle $. 

If the group $D$ is abelian, then 
$D$ is generated by three elements $h(x)=a$, $h(y)=b$ and $h(xy)=c$ such that $a$, $b$ and $c$ commute and one has 
\begin{equation} \label{smallleftalternativeequ2} 
f(x,x)= a l, \  f(y,y)= b l, \  f(xy, xy)= c l, \end{equation} 
where $l=a b c$.   

\medskip 
\noindent 
The Steiner-like loop $L$ is the direct product of $A$ and $S$ precisely if the group $D$ has order $\le 2$.  
\end{Prop}
\begin{proof} 
We assume that condition a) holds. 
The group $D$ is generated by the elements $h(x)$, $h(y)$, $h(xy)$. If $D$ is not abelian, then according to 
Lemma \ref{hxinvolution} the set $\{h(x)^2, h(y)^2, h(xy)^2 \}$ is contained in the centre $Z(D)$ of $D$ and the factor group $D/Z(D)$ is isomorphic to $S_3$. The group $D$ is generated by the elements $h(x)=a$ and 
$h(x) h(y)=t$ such that $a^2 \in Z(D)$ and $t^3 \in Z(D)$. As $a^{-1} t a= t^{-1} z$ with a suitable element $z \in Z(D)$ one has $t^3= a^{-1} t^3 a= t^{-3} z^3$ and hence $z^3=t^6$. Since 
$a^2 \in Z(D)$  we have  $a^{-2n} t^{-r} a^{2n} t^r=1$ and 
$a^{-2n-1} t^{-r} a^{2n+1} t^r= a^{-1} t^{-r} a t^r=t^{2r} z^{-r}$ for all $n, r \in \mathbb N$. As $t^3 \in Z(D)$ one has 
$a^{-1} t^{-3r} a t^{3r}= 1$, $a^{-1} t^{-3r-1} a t^{3r+1}= t^2 z^{-1}$, $a^{-1} t^{-3r-2} a t^{3r+2}= t^4 z^{-2}$ for all $r \in \mathbb N$. 
Hence the commutator subgroup $D'$ of $D$ generated by $t^2 z^{-1}=s$ is the cyclic group of order $3$ which is not contained in $Z(D)$. 
Therefore $D' \cap Z(D) = \{ 1 \}$ and the direct product $K=D' \times Z(D)$ is a normal subgroup of index $2$ in $D$. Since $a^2 \in Z(D)$ the 
element $a$ induces on $K$ an involution fixing $Z(D)$ elementwise and acting on $\langle s \rangle $ by $a^{-1} s a= s^{-1}$. Hence 
$D=K \cdot \langle a \rangle $ with $K= \langle s \rangle \times Z(D)$ such that $Z(D)$ is generated by $a^2$ and $t^3$. 

Conversely, the element $h(xy)$ has the form $h(xy)= t a z$, $z \in Z(D)$. Using $h(y)=a^{-1} t$ and putting $f(x,x)=a^2 z$, $f(y,y)=t^3 z$, 
$f(xy,xy)= t^3 a^2 z$ we see that identity (\ref{smallequ3}) is satisfied. 
 
If the group $D$ is abelian, then 
putting $h(x)= a$, $h(y) = b$ and $h(xy ) = c$ into identity (\ref{smallequ3}) we obtain identity (\ref{smallleftalternativeequ2}) with 
$l=a b c$. 
If identity (\ref{smallleftalternativeequ2}) holds, then identity (\ref{smallequ3}) is satisfied. 

If the group $D$ is non-abelian, then it contains two non-commuting elements $h(x)=a$ and $h(x)h(y)=t$. By Remark 3.1 the loop $L$ is the direct product of $S$ and $A$ only if $t=1$. Since $a^{-1} t a= t^{-1} z$ with a suitable element 
$z \in Z(D)$ and $a \notin Z(D)$ we get  a contradiction. Hence the group $D$ is abelian. In this case the loop $L$ is the direct product of $S$ and $A$ if and only if $a=b=c$ and for the order $o(a)$ of $a$ one has $o(a) \le 2$. This proves the last assertion.  
\end{proof}

\begin{Prop} \label{Prop4} 
Let $S$ be a Steiner loop and $A$ be an abelian group. Let $f$ be a mapping  $S \times S \to A$ such that for 
$x \neq y \in S \setminus \{ e \}$ one has $f(x,y)=h(x) h(y)$ for a mapping $h: S \setminus \{ e \} \to A$ and 
$f(x,e)=f(e,x)=1$ for all $x \in S$. Then the following assertions are equivalent: 
\newline
\noindent
a) For all $x, y \in S \setminus \{ e \}$ and $x \neq y$ the identity 
\begin{equation} \label{equ13} h(x)^2 h(y)^2  = f(x,x) f(y,y) f(x y, x y)^{-1} \end{equation} 
holds. 
\newline
\noindent
b) If $S$ has more than $4$ elements, then the elements $h(x)$ have the form $h(x)=u \omega _x$, where $u$ is a   
fixed element of $A$ and for the order $o(\omega _x )$ of the element $\omega _x $  
 one  has  $o(\omega _x ) \le 2$.  The  elements $f(x,x)$ have  the   
 form $f(x,x)=u^4 \rho _x$ such that for the order $o(\rho _x )$ of the element $\rho _x$  
 one  has $o(\rho _x) \le 2$  and  $\rho _x \rho _y = \rho _{xy}$ for  all $x \neq y \in S \setminus \{ e \}$.  
Let $\Gamma $ be the set of the different involutions in $\{ \omega _x , x \in S \setminus \{ e \} \}$. 
The group $D \langle u \rangle $ is the direct product $\langle u \rangle \times \langle \Gamma \rangle $. 

If $S=\{e, x, y, xy \}$ is the elementary abelian group of order $4$, then one has 
\begin{equation} \label{smallautomorphicequ1} h(x) =  \alpha ,\ \ h(y) = \beta ,\ \ h(x y) =\gamma, \end{equation} 
\begin{equation} f(x,x) = \alpha^2 \beta \gamma \rho_x, \  \ f(y,y) = \alpha \beta^2 \gamma \rho_y,\ \ 
f(xy, xy) = \alpha \beta \gamma^2 \rho_x \rho_y, \nonumber \end{equation}  
where $\alpha, \beta, \gamma $ are elements of $A$, whereas the orders of $\rho_x, \rho_y $ are at most 2.

\medskip
If the Steiner loop $S$ has more than four elements, then 
the Steiner-like loop $L$  is the direct product of $A$ and $S$ precisely if the group $D$ is the direct product $\langle u \rangle \times \langle \omega \rangle $ such that for the orders $o(u)$ as well as $o(\omega )$ one has $o(u) \le 2$ and $o(\omega ) \le 2$ and 
$\rho _x=1$ for all $x \in S \setminus \{ e \}$.  If $S=\{e, x, y, xy \}$ is the elementary abelian group of order $4$, then  $L$  is the direct product of $A$ and $S$ precisely if the group $D$ has order $\le 2$ and $\rho _x=1$ for all 
$x \in S \setminus \{ e \}$.
\end{Prop}
\begin{proof}  
Identity (\ref{equ13}) is the same as identity (\ref{equ10}) in the proof of Theorem \ref{Prop3}.  
If $|S| > 4$, then using (\ref{equ21}) it follows as in the proof of Theorem \ref{Prop3} that 
for all $x \in S \setminus \{ e \}$ one has 
$h(x)= u \omega _x= \omega _x u$, where  $o(\omega _x ) \le 2$. Since $f(x,x)^2= r h(x)^2= u^8$ one has 
$f(x,x)= u^4 \rho_x = \rho_x u^4$, where $o( \rho_x) \le 2$.  
Using this relation identity (\ref{equ13}) yields that $\rho _x \rho _y =\rho _{xy}$ for all 
$x \neq y \in S \setminus \{ e \}$. Since the group $D \langle u \rangle $ is generated by $D$ and $\langle u \rangle $ we have $D \langle u \rangle =\langle u \rangle \times \langle \Gamma \rangle $, where $\Gamma $ is the set of the different involutions in $\{ \omega _x , x \in S \setminus \{ e \} \}$.

If $S$ is the elementary abelian group of order $4$, then identity (\ref{equ21}) yields  
that for all $x \in S \setminus \{ e \}$ one has $f(x,x)^2=r h(x)^2$ with 
$r=h(x)^2 h(y)^2 h(xy)^2$.  Putting $h(x)= \alpha , h(y) = \beta $ and $h(xy ) = \gamma $ into this relation we get 
\[ f(x,x)^2 = \alpha^{4}\beta^2\gamma^2,\ \ f(y,y)^2 = \alpha^2\beta^{4}\gamma^2,\ \  f(xy, xy)^2 = \alpha^2\beta^2\gamma^{4}. \] 
As the group $A$ is abelian it follows that 
\[ f(x,x)= \alpha^{2}\beta \gamma \rho _x,\ \ f(y,y)= \alpha \beta^{2} \gamma \rho _y,\ \  f(xy, xy)= \alpha \beta \gamma^{2} \rho _{xy} \] 
such that the orders of $\rho _x, \rho _y, \rho _{xy}$ are at most $2$. 
Moreover, identity (\ref{equ13}) gives $\rho _x \rho _y= \rho _{xy}$. Hence from assertion a) it follows b). 

Conversely, if assertion b) is satisfied, then identity (\ref{equ13}) holds. Hence a) and b) are equivalent. 

By Remark \ref{rem1} the loop $L$ is the direct product of $A$ and $S$ if and only if $f(x,y)=1$ for all $x,y \in S$. 
First we assume that $|S| > 4$. Since $h(x) h(y)=u^2 \omega _x \omega _y= 1 = h(x) h(xy)=u^2 \omega _x \omega _{xy}$ for all $x \neq y \in S \setminus \{ e \}$ it follows $\omega _x= \omega $ for all $x \in S \setminus \{e \}$ with $o(\omega ) \le 2$ and hence  $u^2=1$.  As $f(x,x)=u^4 \rho _x =1$ for all $x \in S \setminus \{ e \}$ we get $\rho _x =1$ for all 
$x \in S \setminus \{ e \}$.  
 
If $S$ is the elementary abelian group of order $4$, then $f(x,y)=1$ for all $x,y \in S$ precisely if 
$\alpha = \beta = \gamma $ such that $o(\alpha ) \le 2$ and $\rho _x =1$ for all $x \in S \setminus \{ e \}$.  
\end{proof}

The structure of the group $D \langle u \rangle $ in Proposition \ref{Prop4} differs significantly from the structure of this group in Theorem \ref{Prop3} since we can not use Lemma \ref{hxinvolution}.

\section{Elementary properties of Steiner-like loops}

\begin{Prop} \label{flexible} 
A Steiner-like loop $L$ is flexible if and only if the set 
$\{ f(x,y); x, \\y \in S \}$ is contained in the centre of the group $A$. 
\end{Prop} 
\begin{proof} According to Proposition 3.10 in \cite{nagy}, p. 767, the loop $L$ is flexible if and only if 
\begin{equation} \label{flexibleequ} f(xy,x) f(x,y) \xi = f(x,xy) \xi f(y,x) \end{equation} 
holds for all $x,y \in S$ and $\xi \in A$. Putting $\xi =1$ into (\ref{flexibleequ}) one has 
\begin{equation} \label{flexibleequ2} f(xy,x) f(x,y) = f(x,xy) f(y,x).  \nonumber \end{equation} 
Using this relation identity (\ref{flexibleequ}) reduces to $f(y,x) \xi = \xi f(y,x)$. Hence we obtain $f(x,y) \in Z(A)$ for all $x,y \in S$. 
\newline
If $f(x,y) \in Z(A)$ for all $x,y \in S$, then using the function $h$ we obtain $h(x) h(y) \in Z(A)$. This yields that the range of $h$ is commutative (see Lemma \ref{hxcommutative}) and hence identity (\ref{flexibleequ}) is true.  
\end{proof}

\noindent
In a flexible Steiner-like loop $L$ for any element the left inverse and the right inverse coincide and the group $A$ is contained in the nucleus of $L$.

\noindent
\begin{Prop} \label{leftalternative} For a Steiner-like loop $L$ the following properties are equivalent: 
\newline
a) $L$ is right alternative,  
\newline
b) $L$ satisfies the right inverse property,  
\newline
c) the set ${\mathcal F}= \{ f(z, z); \ z \in S \}$ is contained in the centre of $A$ and identity (\ref{smallequ3}) of Theorem \ref{Prop3} holds. 
\end{Prop}
\begin{proof} For all $(x, \xi )$ and $(y, \eta )$ of $L$ identity (\ref{equc4}) holds if and only if the identity
\begin{equation} \label{equ26} \eta ^{-1} f(xy,x) f(y,x) \eta =  f(x,x)    \end{equation} 
is satisfied. If $x=y$ we obtain that the set ${\mathcal F}$ is contained in the 
centre of $A$.  Using for $x \neq y$ the mapping $h$ we obtain with $\eta =h(xy)$ identity (\ref{smallequ3}) of Theorem  \ref{Prop3}. 

The loop $L$ satisfies the right  inverse property if and only if  for all $(x, \xi )$ and $(y, \eta )$ of $L$ identity (\ref{equc6}) holds, where  
$(x, \xi )^{\rho }$ is the right inverse of $(x, \xi )$. Since  
$(x, \xi )^{\rho }=(x, \xi ^{-1} f(x,x)^{-1})$ identity (\ref{equc6}) yields again identity (\ref{equ26}).  
\end{proof}

\noindent
In a right alternative Steiner-like loop $L$ for any element the left inverse and the right inverse coincide. 

\noindent
\begin{Prop} \label{rightalternative} For a Steiner-like loop $L$ the following properties are equivalent: 
\newline
a) $L$ is left alternative,  
\newline
b) $L$ satisfies the left inverse property,  
\newline
c) the sets ${\mathcal F}= \{ f(z, z); \ z \in S \}$ and ${\mathcal K}=\{h(x) h(y); \ x, y \in S \setminus \{ e \}, x \neq y \}$ are contained in the centre of $A$ and identity 
(\ref{smallequ3}) of Theorem \ref{Prop3} is satisfied.    
\end{Prop}
\begin{proof} For all $(x, \xi )$ and $(y, \eta )$ of $L$  identity (\ref{equc2}) holds if and only if the identity
\begin{equation} \label{equc5}  f(x, xy) \xi f(x,y)= f(x,x)  \xi  \end{equation} 
is satisfied. Putting in this identity  $x=y$ it follows that the set 
${\mathcal F}$ is contained in the centre of $A$.    
Using for $x \neq y$ the mapping $h$ we obtain 
\begin{equation}  \label{equ27}  h(x) h(xy) \xi h(x) h(y) \xi^{-1} =f(x,x). \end{equation}
For any $x$ and $y$, $x \neq y$, the product $h(x) h(y)$ is contained 
in the centre of $A$ and hence identity (\ref{equ27}) reduces to identity (\ref{smallequ3}) of 
Theorem \ref{Prop3}. 

The loop $L$ satisfies the left  inverse property if and only if  for all $(x, \xi )$ and $(y, \eta )$ of $L$ identity (\ref{equc3}) holds, where 
$(x, \xi )^{\lambda }$ is the left inverse of $(x, \xi )$. Since  
$(x, \xi )^{\lambda }=(x, f(x,x)^{-1} \xi ^{-1})$ identity (\ref{equc3}) yields   
\begin{equation} \label{equequ1} \xi ^{-1} f(x,y) \xi f(x,xy) = f(x,x) \end{equation} 
for all $x, y \in S$ and $\xi \in A$.  Therefore the set 
$\{ f(x, x); \ x \in S \}$ as well as all products $h(x) h(y)$, $x, y \in S \setminus \{ e \}$ with 
$x \neq y$ are contained in the centre of $A$. Hence identity (\ref{equequ1}) reduces to identity (\ref{smallequ3}). 
\end{proof}

\noindent
A Steiner-like loop $L$ which has the left inverse property has the inverse property (Propositions 
\ref{leftalternative}, \ref{rightalternative}). In a left as well as a right alternative Steiner-like loop  the values $f(x,x)$ are determined by the function $h$. Any  left alternative Steiner-like loop $L$ is also right alternative.

\noindent
\begin{Prop} \label{crossinverse} A Steiner-like loop $L$ satisfies the cross  inverse property if and only if the group $A$ is commutative and identity (\ref{smallequ3}) of Theorem \ref{Prop3} holds.    
\end{Prop}
\begin{proof} The Steiner-like loop $L$ satisfies the cross inverse property if and only for all $(x, \xi )$ and $(y, \eta )$ of $L$ identity (\ref{equk1}) holds, where $(x, \xi )^{\rho }$ is the right inverse of $(x, \xi )$. Since  
$(x, \xi )^{\rho }=(x, \xi ^{-1} f(x,x)^{-1})$  identity (\ref{equk1}) gives 
\[ f(xy,x) f(x,y) \xi \eta \xi^{-1} f(x,x)^{-1}= \eta \] 
for all $x, y \in S$ and $\xi, \eta \in A$. Putting $x=e$ it follows that the group $A$ is commutative. Using the function 
$h$ we obtain $f(x,x)= h(x) h(y) h(x) h(xy)$ which yields the assertion. 
\end{proof}

\noindent
Any Steiner-like loop $L$ satisfying the cross inverse property is commutative, alternative and has the inverse property. 

If $L$ is right alternative or it has the right inverse property and the group $D$ generated by the set 
$\{ h(x), x \in S \setminus \{ e \} \}$ is not abelian, then $L$ cannot have the left inverse property, one has 
$h(x)=u \omega _x$ with $o(\omega _x) \le 2$ and the factor group 
$D \langle u \rangle / \langle u \rangle $ is a restricted Fischer group. If the group $D$ is abelian, then the Steiner-like loop $L$ has neither the left inverse nor the left alternative property if and only if the set 
${\mathcal K}=\{h(x) h(y); \ x, y \in S \setminus \{ e \}, x \neq y \}$ is not contained in the centre of $A$.

\begin{Lemma} \label{Steinergroup} A Steiner loop $S$ is a Bol loop if and only if $S$ is an elementary abelian $2$-group. 
\end{Lemma} 
\begin{proof} Since a commutative right respectively left Bol loop is Moufang (cf. \cite{flugfelder}, p. 121) and any Moufang loop of exponent $2$ is an elementary abelian $2$-group (cf. Proposition 2 in \cite{chein}, p. 35) the assertion follows. 
\end{proof}

\begin{Theo} \label{groupuj} Let $L$ be a Steiner-like loop such that $S$ has more than $4$ elements.  Then the following conditions are equivalent: 
\newline
(i) $L$ is a group.  
\newline
(ii) $S$ is an elementary abelian $2$-group and for all $x \in S \backslash \{ e \}$ one has $h(x)=t$ and $f(x,x)=t^4$, where $t$ is a fixed element of $A$ such that $t^2 \in Z(A)$. 
\newline 
(iii) $L$ is the direct product of the group $A$ and an  abelian group $C$ with amalgamated  cyclic central subgroup $\Delta $ such that the factor group 
$S \cong C/ \Delta $ is an elementary abelian $2$-group.   
\newline
(iv) $L$ is a left Bol loop. 
\end{Theo} 
\begin{proof} First we assume that $L$ satisfies the associative law. This is the case precisely if $S$ is an elementary abelian $2$-group and 
\begin{equation} \label{equf1} f(xy,z) f(x,y) \xi = f(x,yz) \xi f(y,z) \end{equation}  
holds for all $x,y,z \in S$ and $\xi \in A$. Putting $x=e$ into identity (\ref{equf1}) one has $f(y,z) \in Z(A)$ for all $y, z \in S$. 
Hence (\ref{equf1}) reduces to \begin{equation} \label{equf2} f(x,y) f(xy,z)=f(y,z) f(x,yz). \nonumber \end{equation}  
Since the range of $h$ is commutative (cf. Lemma \ref{hxcommutative}) this identity  yields that $h(x)=t$ for all 
$x \in S \setminus \{ e \}$. For $x=y$ we get $f(x,x)=h(x) h(z) h(x) h(xz) =t^4$ for all $x \in S \setminus \{ e \}$ and 
$h(x) h(y)=t^2 \in Z(A)$.  Hence property (i) yields property (ii). 
Using the {\bf Construction} in Section 3 from property (ii) we obtain the assertion (iii). Now we assume that (iii) holds. The multiplication of $L$ on the set 
$S \times B \times \Delta $  given by (\ref{equ2}) is associative if and only if 
\begin{equation} \label{construction} k(\rho _1 \rho _2, \rho _3) k( \rho _1, \rho _2) = k(\rho _1, \rho _2 \rho _3) k( \rho _2, \rho _3)  \nonumber  \end{equation}  
holds for all $\rho _1, \rho _2, \rho _3 \in B$, where the map $k: B \times B \to \Delta $ is a factor system for the group $B$. From identity (1) of Satz 14.1 in \cite{huppert} (p. 86) we obtain that this identity  is true and (i) follows.  
\newline
Clearly from (i) it follows (iv).  It remains to prove that if $L$ is a left Bol loop, then $L$ is a group. 
If $L$ is a left Bol loop, then $S$ is an elementary abelian $2$-group (see Lemma \ref{Steinergroup}). Moreover, $L$ is left  alternative, i.e. the sets 
$\{ f(x,x);  \ x \in S \}$, $\{h(x) h(y); \ x, y \in S \setminus \{ e \}, x \neq y \}$ are contained in the centre of $A$ and identity (\ref{smallequ3}) of Theorem \ref{Prop3} holds (see Proposition \ref{rightalternative}). Hence every left Bol loop is right alternative 
(cf.  Proposition \ref{leftalternative}) and therefore it is a Moufang loop (cf. Theorem IV.6.9 in \cite{flugfelder}, 
p. 116). Then $L$ satisfies the right Bol identity, i.e. for all $(x, \alpha )$, $(y, \beta )$, $(z, \gamma )$ of $L$ identity (\ref{equk4}) holds,    
which reduces to  
\begin{equation} \label{Moufangequ2} h(xyz)= h(y) h(xy) h(xz)^{-1}. \end{equation} 
Putting $z=y$ into (\ref{Moufangequ2}) we obtain $h(x)=h(y)$ and hence $h(x)=t \in A$ for all 
$x \in S \setminus \{ e \}$ and $f(x,x)=t^4$, where $t$ is a fixed element of $A$ such that $t^2 \in Z(A)$ and the assertion (ii) follows. This completes the proof of the theorem. 
\end{proof}

\begin{Theo} \label{smallgroup} Let $L$ be a Steiner-like loop  such that $S=\{ e, x, y, xy \}$ is the elementary abelian group of order $4$.  Then the following conditions are equivalent: 
\newline
(i) $L$ is a group. 
\newline
(ii) $L$ is left alternative or it has  the left inverse property.  
\newline
(iii) $L$ is a left Bol loop. 
\newline 
Moreover, $L$ is a commutative group if and only if $L$ has the cross inverse property. 
\end{Theo} 
\begin{proof} Since every left Bol loop is left alternative we have only to prove that (ii) implies (i). 
A Steiner-like loop $L$ is left alternative if and only if the sets $\{ f(x,x); \ x \in S \}$, $\{ h(x) h(y); x \neq y \in S \setminus \{ e \} \}$ are contained  in the centre of $A$ and identity (\ref{smallequ3}) 
of Theorem \ref{Prop3} holds (cf. Proposition  \ref{rightalternative}).
Using this if $S$ has only $4$ elements, then identity (\ref{equf1}) trivially holds.  
 
According to Proposition \ref{crossinverse} the loop $L$ has the cross inverse property precisely if $L$ is left alternative and the group $A$ is commutative. In this case $L$ is abelian (see Section 3, p. 5) and the last assertion follows. \end{proof}

\begin{Prop} \label{smallleftbol} Let $L$ be a proper Steiner-like loop.  
\newline
\noindent
a) If $S=\{ e, x, y, xy \}$ is the elementary abelian group of order $4$, then  
the loop $L$ satisfies the right Bol identity if and only if the range of $h$ is commutative, the set $\{ f(x,x);\ x \in S \}$ is contained in the centre $Z(A)$ of $A$ and identity (\ref{smallequ3}) of Theorem \ref{Prop3} holds.  
\newline
\noindent
b) If the Steiner loop $S$ has more than $4$ elements, then the loop $L$ is a right Bol loop precisely if $S$ is an elementary abelian $2$-group and for all $x \in S \backslash \{ e \}$ one has $h(x)=t$ and $f(x,x)=t^4$ with a fixed element $t$ of $A$, but $t^2 \notin Z(A)$. 
 \end{Prop}
\begin{proof}  The loop $L$ is a right Bol loop if and only if the Steiner loop $S \cong L/A$ is an elementary abelian $2$-group (cf. Lemma \ref{Steinergroup}) and for all $x,y,k \in S$ and $\alpha  \in A$ identity 
\begin{equation} \label{leftbolequ1} f(k,y) \alpha f(xy,x) f(x,y)= f(xyk,x) f(xk,y) f(k,x) \alpha \end{equation} 
holds. 
Since every right Bol loop is right alternative, the set $\{ f(x,x); \ x \in S \}$ is contained  in the centre of $A$ and identity (\ref{smallequ3}) 
of Theorem \ref{Prop3} holds (cf. Proposition  \ref{leftalternative}). Using the function $h$, putting  $k=e$ and 
$\alpha =h(xy)$ into (\ref{leftbolequ1}) we have  
\begin{equation} \label{leftbolequ3} 
h(x)^2 h(y) h(xy) =  h(xy) h(x)^2 h(y)  \end{equation} 
for all $x \neq y \in S \setminus \{ e \}$. According to Lemma \ref{hxinvolution} a) we have 
$h(xy) h(x)^2= h(x)^2 h(xy)$ for all $x,y \in S \setminus \{ e \}$, $x \neq y$. Hence identity (\ref{leftbolequ3}) reduces to 
\begin{equation} \label{leftbolequ4} 
h(xy) h(y) =  h(y) h(xy) \nonumber \end{equation} 
for all $x \neq y \in S \setminus \{ e \}$ which gives that the range of $h$ is commutative. Therefore if $|S|=4$ we have assertion a). 
\newline
\noindent
Now we assume that $|S| > 4$. 
Interchanging $x$ and $k$ into identity (\ref{leftbolequ1}) we get 
\begin{equation} \label{leftbolequ12} f(x,y) \alpha f(ky,k) f(k,y) = f(xyk, k) f(xk,y) f(x,k) \alpha \end{equation} 
for all $x,y,k \in S$ and $\alpha  \in A$.   
As the range of $h$ is commutative identity (\ref{leftbolequ1}) reduces to $h(y) h(xy)= h(xk) h(xyk)$ and identity (\ref{leftbolequ12}) reduces to 
$h(y) h(ky)= h(xk) h(xyk)$ for all $x, y, k \in S \setminus \{ e \}$, $x \neq y$, $x \neq k$, 
$y \neq k$. Comparing the last two identities we get $h(xy)= h(ky)$ or equivalently $h(x)=t$ and $f(x,x)=t^4$ with a fixed element 
$t \in A$ for all $x \in S \setminus \{ e \}$. 

Conversely, if $h(x)=t$ and $f(x,x)=t^4$, but $t^2 \notin Z(A)$, then identity (\ref{leftbolequ1}) holds
for all $x,y,k \in S$ and $\alpha  \in A$. Hence $L$ satisfies the right Bol identity. \end{proof}

\noindent
A Steiner-like right Bol loop $L$ does not need to be left alternative as the following examples show. 
Let $A$ be a finite simple group which is not isomorphic to one of the following groups: $PSL(2, 2^n)$, $n>1$, $PSL(2,q)$, $q \equiv 3$ or $5$ \  $(\hbox{mod} \ 8)$, $q > 3$, the Janko group of order 175560, a group of Ree type of order $q^3 (q-1) (q^3+1)$, where $q=3^{2k+1}$, $k \ge 1$. Then $A$ contains an element $u$ of order $4$ (cf. \cite{walter}, Theorem I and \cite{huppert}, Kapitel II, Satz 8.2 and 8.10). Putting $h(x) = u$ for all  
$ x \in S \setminus \{ e \} $ we obtain a right Bol loop $L$ which is not left alternative and has exponent $4$ since $f(x,x)=1$.

\noindent
\begin{Prop} \label{automorphicinverse} A Steiner-like loop $L$ has the automorphic inverse property if and only if the group $A$ is commutative and identity (\ref{equ13}) of Proposition \ref{Prop4} holds.  
\end{Prop}
\begin{proof} As $x^{-1}=x$ for all $x \in S$ Proposition 3.5 in \cite{nagy}, p. 764, yields that the Steiner-like loop $L$ satisfies identity (\ref{equk2}) if and only if the group $A$ is commutative and identity  
\begin{equation} \label{equ33} f(y,y) f(x,x) f(x y, x y)^{-1} =  f(x,y)^2 \nonumber 
\end{equation} 
holds. Using the function $h$ we obtain identity (\ref{equ13}) of Proposition \ref{Prop4} and the assertion follows. 
\end{proof}

\noindent
Any Steiner-like loop $L$ satisfying  the automorphic inverse property is commutative. 
Every Steiner-like loop $L$ having the cross inverse property satisfies the automorphic inverse property. There are loops $L$ having the automorphic inverse property but not the cross inverse property. This is the case if one chooses in identities  (\ref{smallautomorphicequ1}) of Proposition \ref{Prop4} for an element 
$x \in S \setminus \{ e \}$ the element $\rho_x \neq 1$.

\noindent 
\begin{Prop} \label{weekinverse} A Steiner-like loop $L$ has the weak inverse property if and only if the set 
$\{ f(x,y); \ x, y \in S \}$ is contained in the centre of $A$ and for all $x \in S \setminus \{ e \}$ one has $f(x,x)=s h(x)=h(x) s$, where $s$ 
is a fixed element of $A$. 
\end{Prop} 
\begin{proof} First we assume that identity (\ref{equweek}) holds, which means that 
\begin{equation} \label{weekequuj1} (x, \xi ) [(y, \eta ) (z, \alpha )]= (x (yz), f(x,yz) \xi f(y,z) \eta \alpha )=(e,1)  \end{equation}
is satisfied precisely if 
\begin{equation} \label{weekequuj2} [(x, \xi ) (y, \eta )] (z, \alpha )= ( (x y)z, f(xy,z) f(x,y) \xi \eta \alpha )=(e,1) \end{equation}
holds.   
Clearly $x(yz)=e$ if and only if $(xy) z=e$. Hence we can write $z=xy$. Therefore identity 
\begin{equation} \label{equfirstweak} f(x,x) \xi f(y,xy) \eta \alpha =1 \end{equation} 
is satisfied whenever identity  
\begin{equation} \label{equsecondweak} f(xy,xy) f(x,y) \xi \eta \alpha = 1 \end{equation}
holds. The set $\{ f(x,y); \ x,y \in S \}$ is contained in the centre of $A$ since for loops having the weak inverse property the left, the middle and the right nucleus coincide (cf. Theorem 1 in \cite{osborn}, p. 297). If identities (\ref{equfirstweak}) and (\ref{equsecondweak}) are equivalent,  then the identity 
\begin{equation} \label{equthird} f(x,y) f(xy,xy)= f(y,xy) f(x,x) \nonumber \end{equation}
holds for all $x, y \in S$. 
Using the function $h$ this relation yields  
\begin{equation} \label{equfour} f(xy,xy) h(xy)^{-1}= f(x,x) h(x)^{-1} \nonumber \end{equation} 
for all $x, xy \in S \setminus \{ e \}$. Hence for all $x \in S \setminus \{ e \}$ one has 
$f(x,x)=s h(x)=h(x) s$ 
with a fixed element $s$ of $A$. 

Conversely, now we assume that for all $x,y \in S$ the set $\{ f(x,y); \ x, y \in S \}$ is contained in the centre 
of $A$ and for all $x \in S \setminus \{ e \}$ one has $f(x,x)=s h(x)$. We have to prove the equivalence of identities (\ref{weekequuj1}) and 
(\ref{weekequuj2}). For Steiner loops the property $x(yz)=e$ if and only if $(xy) z=e$ is always true. Hence it remains to prove the equivalence of  identities (\ref{equfirstweak}) and (\ref{equsecondweak}). This holds since  
\begin{equation} 1=f(x,x) \xi f(y,xy) \eta \alpha = f(x,x) f(y,xy) \xi \eta \alpha = \nonumber \end{equation}
\begin{equation} s h(x) h(y) h(xy) \xi \eta \alpha = s h(xy) h(x) h(y) \xi \eta \alpha =f(xy,xy) f(x,y) \xi \eta \alpha = 1. \nonumber \end{equation}  
\end{proof}

\noindent
Every Steiner-like loop $L$ which is alternative or which has the inverse property has the weak inverse property. 
There are Steiner-like loops having the weak inverse property but which are neither left alternative nor have the left inverse property. This is the case if $s \neq h(x) h(y) h(xy)$. 

\begin{Coro} Let $L$ be a Steiner-like loop  such that $S=\{ e, x, y, xy \}$ is the elementary abelian group of order $4$. 
\newline
\noindent
a) If $L$ is right alternative or it has the right inverse property or it has the weak inverse property, then $L$ is a group if and only if $L$ is alternative or it has the inverse property. 
\newline
\noindent
b) If $L$ has the automorphic inverse property, then $L$ is an abelian group if and only if $L$ has the cross inverse property. 
\end{Coro}
\begin{proof}  A loop $L$ is a group precisely if $L$ is left alternative (see Theorem \ref{smallgroup}). Then the assertion follows from Propositions \ref{leftalternative}, \ref{rightalternative}, \ref{crossinverse}, \ref{automorphicinverse}, \ref{weekinverse}. \end{proof}

\section{Variations of extension process (\ref{equextension})} 
A loop extension of type (\ref{equextension}) with respect to the factor system $f$ may be also defined 
in such a way that the multiplication on the set $S \times A$ is given by
\[(x, \xi )(y, \eta ) = (xy, \xi f(x, y) \eta), \] respectively 
\[(x, \xi )(y, \eta ) = (xy, \xi \eta f(x, y)). \]
This yields loops $L^*$, respectively  $L^{**}$ which coincide with the extension given by multiplication (\ref{equextension})   
if and only if the set $\{f(x,y); \ x,y \in S \}$ is contained in the centre of $A$. Therefore, we have to compare loops $L$,  $L^*$, $L^{**}$ for which the set $\{f(x,y); \ x,y \in S \}$ is not contained in the centre of $A$.

The left inverse of the element $(x, \xi)$ in $L^*$, respectively in $L^{**}$ has the form 
$(x^{-1}, \xi ^{-1} f(x^{-1},x)^{-1} )$, respectively 
$(x^{-1},  f(x^{-1},x)^{-1}\xi ^{-1} )$.
The right inverse of $(x, \xi)$ in $L^*$, respectively in $L^{**}$ is $(x^{-1}, f(x,x^{-1})^{-1} \xi ^{-1})$, respectively $(x^{-1}, \xi ^{-1} f(x,x^{-1})^{-1})$. Here is $x^{-1}= x \backslash e= e /x$ in $S$. \\
Hence for any element the left inverse coincides with the right inverse in a loop $L$, $L^{\ast }$ and $L^{\ast \ast }$ if and only if 
$f(x^{-1},x)= \xi ^{-1} f(x,x^{-1}) \xi $ for all $x \in S$ and $\xi \in A$.

\noindent
For $L^*$ the group $A$ is contained in the left and in the right nucleus of $L^*$, for $L^{**}$ the group $A$ is contained in the left and in the middle nucleus of $L^{**}$. The group $A$ is never contained in the nucleus of $L^*$ as well as of $L^{**}$.

\noindent
If for a loop $L^{\ast }$ the set $\{f(x,y); \ x,y \in S \}$ is not contained in the centre of $A$, then one checks easily that $L^{\ast }$ does not satisfy the following identities: (\ref{equc1}), (\ref{equc2}), (\ref{equc4}), (\ref{equc3}), (\ref{equc6}), (\ref{equk1}), (\ref{equk2}), (\ref{equweek}), (\ref{equk3}), (\ref{equk4}). Hence any such loop $L^{\ast }$ is far from the associativity.  
\newline
\noindent
If for a loop $L^{\ast \ast }$ the set $\{f(x,y); \ x,y \in S \}$ is not contained in the centre of $A$, then it follows that $L^{\ast \ast }$ does not satisfy the identities: (\ref{equc1}), (\ref{equc4}), (\ref{equc6}), (\ref{equk1}), (\ref{equk2}), (\ref{equweek}), (\ref{equk4}). Now we determine the conditions under which a loop $L^{\ast \ast }$ satisfies conditions 
(2.2), (2.4) and (2.9).  

\begin{Prop} \label{uj rightalternative} 
A Steiner-like loop $L^{\ast \ast }$ is left alternative or it satisfies the left inverse property precisely if  
the set ${\mathcal F}= \{ f(z, z); \ z \in S \}$ is contained in the centre of $A$ and identity (\ref{smallequ3}) in Theorem  \ref{Prop3} holds.   
\end{Prop} 
\begin{proof} 
The loop $L^{\ast \ast }$ satisfies identity (\ref{equc2}) as well as identity (\ref{equc3}) if and only if 
\begin{equation} \label{equleft3}  \eta f(x, y) f(x,xy) \eta ^{-1}= f(x,x) \nonumber \end{equation} 
holds for all $x,y \in S$, $\eta \in A$.  
With $x=y$ we obtain  that $\{ f(x,x); \ x \in S \}$ is contained in the centre of $A$. Using the function $h$ for $x \neq y$ this identity  yields 
identity (\ref{smallequ3}) in Theorem \ref{Prop3} and the assertion follows. \end{proof}

\begin{Prop} 
Let $L^{\ast \ast }$ be a proper Steiner-like loop. 
\newline
\noindent
a) If $S=\{ e, x, y, xy \}$ is the elementary abelian group of order $4$, then  
the loop $L^{\ast \ast }$ satisfies the left Bol identity if and only if the range of $h$ is commutative, the set 
$\{ f(x,x);\ x \in S \}$ is contained in the centre of $A$ and identity (\ref{smallequ3}) of Theorem \ref{Prop3} holds.  
\newline
\noindent
b) If the Steiner loop $S$ has more than $4$ elements, then the loop $L^{\ast \ast }$ is a left Bol loop precisely if $S$ is an elementary abelian $2$-group and for all $x \in S \backslash \{ e \}$ one has $h(x)=t$ and $f(x,x)=t^4$ with a fixed element $t$ of $A$, but $t^2 \notin Z(A)$. 
\end{Prop}
\begin{proof}  The loop $L^{\ast \ast }$ is a left Bol loop if and only if the Steiner loop $S \cong L/A$ is an elementary abelian $2$-group 
(cf. Lemma \ref{Steinergroup}) and for all $x,y,k \in S$ and $\alpha  \in A$ the identity 
\begin{equation} \label{equc22} f(y,x) f(x, xy) \alpha f(y,k) = \alpha f(x,k) f(y, x k) f(x, x y k) \end{equation} 
holds. 
Since every left Bol loop is left alternative, the set $\{ f(x,x); \ x \in S \}$ is contained  in the centre of $A$ and identity (\ref{smallequ3}) 
of Theorem \ref{Prop3} holds (cf. Proposition  \ref{uj rightalternative}). Using the function $h$ and Lemma \ref{hxinvolution} a), putting  $k=e$ and $\alpha = h(xy)^{-1}$ into (\ref{equc22}) we obtain that the range of $h$ is commutative. Therefore, if $|S|=4$ we obtain assertion a). 
\newline
\noindent
Now we assume that $|S| > 4$. Interchanging $x$ and $k$ in (\ref{equc22}) we get 
\begin{equation} \label{equcuj22} f(y,k) f(k, ky) \alpha f(y,x) = \alpha f(k,x) f(y, x k) f(k, x y k) \end{equation} 
for all $x,y,k \in S$, $\alpha \in A$. The same consideration as in the proof of Proposition \ref{smallleftbol} yields that 
 $h(x)= t$ and $f(x,x)=t^4$ with a fixed element 
$t \in A$ for all $x \in S \setminus \{ e \}$. 

Conversely, if $h(x)=t$ and $f(x,x)=t^4$, but $t^2 \notin Z(A)$, then identity (\ref{equc22}) holds
for all $x,y,k \in S$ and $\alpha  \in A$. Hence $L$ is a left Bol loop. \end{proof}

\section{Groups generated by translations} 
Let $L$ be a Steiner-like loop. 
A left translation $\lambda_{(a, \alpha )}: L \to L$ is the  bijection
\[(x, \xi ) \longmapsto (a, \alpha )(x, \xi ) = (ax, f(a, x) \alpha \xi ) \]
and  a right translation $\rho_{(a, \alpha )}: L \to L$ is the bijection
\[(x, \xi ) \longmapsto (x, \xi )(a, \alpha ) = (ax, f(x, a) \xi \alpha ). \]
\noindent
The  group $G_l$ generated by all left translations of $L$ coincides with the group $G_r$ generated by all right translations of $L$ if and only if 
the group $A$ is abelian. Namely if the element $\alpha $ is not contained in the centre of $A$ and $\rho_{(e, \alpha )}$ is a product of left translations of $L$, then we have $\alpha \xi =\xi \alpha '$ or $\xi ^{-1} \alpha \xi = \alpha '$ 
for all $\xi \in A$ with a suitable $\alpha '$. Hence $\alpha = \alpha '$ and $A$ is commutative. 

For all $\alpha \in A$ the map $\iota_{\alpha}:\ \ (x, \xi ) \longmapsto (x, \alpha^{-1} \xi \alpha )$ is an automorphism of $L$ if and only if 
the  set $\{f(x, y); x, y \in S\}$ is contained in the centre of $A$. 

\begin{Prop} \label{groupalltranslation} Let  $L$ be a Steiner-like loop  such that any $\iota_{\alpha}$ is an automorphism of $L$ and 
let $\Lambda$ be the group generated by the set $\{ \iota_{\alpha}, \alpha \in A \}$.
Then the  group $G$ generated by all  
translations of $L$ coincides with the group generated  by $G_l$ and $\Lambda$.
\end{Prop} 
\begin{proof} One has $\rho_{(a, \alpha )} = \iota_{\alpha } \lambda_{(a, \alpha )}$ if and only if 
$f(x, a) = f(a, x)$ for all $ x, a \in S$ which is true since $\iota_{\alpha }$ is an automorphism of $L$. 
\end{proof}

\noindent 
One has 
$$[\rho_{(a, \gamma)}^{-1}{\rho_{(e, \alpha)}}\rho_{(a, \gamma)}] (x, \xi) 
= (x, \xi \gamma \alpha \gamma^{-1} ) = \ \rho_{(e, \gamma \alpha \gamma^{-1})} (x, \xi ).$$
As $\rho _{(a, \alpha )}= \rho _{(e, \alpha )} \rho _{(a,1)}$  it follows:  

\begin{Prop} \label{lefttranslation} Let $L$ be a Steiner-like loop. Then the  group $G_r$ generated by all right translations of $L$ is an extension of $A$ by the group $\Sigma $ generated by the set $\{\rho _{(a,1)};\ a \in S \}$.
\end{Prop}

\noindent
For the left translations of $L$ one has  
$$[\lambda_{(a, \gamma)}^{-1} \lambda_{(e, \alpha)} \lambda_{(a, \gamma)}] (x, \xi) = 
(x, \gamma ^{-1} f(a,x)^{-1} \alpha  f(a,x) \gamma \xi ).$$ 
To obtain an analogous result for the group $G_l$ generated by all left translations of $L$ we must suppose 
that the set $\{ f(x,y);\ x,y \in S \}$ is contained in the centre of $A$. 
In this case we get $[\lambda_{(a, \gamma)}^{-1} \lambda_{(e, \alpha)} \lambda_{(a, \gamma)}] = \lambda_{(e,\gamma^{-1} \alpha \gamma )}$.  
From the fact that $\lambda _{(a, \alpha )}= \lambda _{(a,1)} \lambda _{(e, \alpha )} $ as well as from Propositions \ref{flexible},    \ref{rightalternative}, \ref{crossinverse}, \ref{automorphicinverse}, \ref{weekinverse}, \ref{groupalltranslation}   it follows:

\begin{Prop} \label{righttranslation} Let $L$ be a Steiner-like loop. We assume that $L$ has one of the following properties:
\newline
\noindent
a) $L$ is flexible, 
\newline
\noindent
b) $L$ is left alternative or it has the left inverse property,  
\newline
\noindent
c) $L$ has the cross inverse property, 
\newline
\noindent
d) $L$ has the automorphic inverse property, 
\newline
\noindent
e) $L$ has the weak inverse property.  
\newline
\noindent
Then the  group $G$ generated by all translations of $L$ coincides with the group generated  by $G_l$ and $\Lambda $.  
Moreover, the  group $G_l$ is 
an extension of $A$ by the group $\Sigma '$ generated by the set $\{\lambda _{(a,1)};\ a \in S \}$.
\end{Prop}

\noindent
As an immediate consequence of Propositions \ref{lefttranslation} and \ref{righttranslation} the following holds.   

\begin{Prop} Let $L$ be a commutative Steiner-like loop. Then the  group $G$ generated by all translations of $L$ is 
a central extension of $A$ by the group $\Sigma '$ generated by the set $\{\lambda _{(a,1)};\ a \in S \}$.
\end{Prop}

\noindent 
The group  generated by the translations of a finite Steiner loop  $S$ of order $n > 3$  contains often the alternating group of degree $n$. This is for instance 
the case if the order of any product of two different translations of the Steiner triple system corresponding to $S$ is odd (cf. \cite{stuhl}, Theorem).

\section{Isomorphisms}   
\noindent 
Let $(L_1, \cdot)$ and $(L_2, \ast )$ be two loops of type (\ref{equextension}) realized on the set $S \times A$
by the multiplications $(x, \xi ) \cdot (y, \eta )=(xy, f_1(x,y) \xi \eta )$ respectively 
$(x, \xi ) \ast (y, \eta )=(xy, f_2(x,y) \xi \eta )$. We consider isomorphism $\alpha :L_1 \to L_2$ such that  
$(e, \xi)^{\alpha }=(e, \xi^{\alpha ''})$. 
For such an isomorphism $\alpha :L_1 \to L_2$ one has $(x,1)^{\alpha }=(x^{\alpha '}, \rho (x^{ \alpha '}))$, where 
$\rho $ is a map from $S$ onto $A$. Using this we have  
$$(x^{\alpha '}, \xi ^{\alpha ''} \rho (x^{\alpha '})) = (e, \xi)^{\alpha } \ast (x,1)^{\alpha } = 
[(e, \xi ) \cdot (x,1)]^{\alpha } =(x, \xi)^{\alpha }= $$ 
$$[(x,1) \cdot (e, \xi )]^{\alpha }= (x,1)^{\alpha } \ast (e, \xi)^{\alpha }= 
(x^{\alpha '}, \rho (x^{ \alpha '})) \ast (e, \xi^{\alpha ''})=(x^{\alpha '}, \rho (x^{\alpha '})
\xi^{\alpha ''}).$$  Hence  every element $\rho (x^{\alpha })$ lies in the centre of $A$. Moreover, one has 
\begin{equation} ((x y)^{\alpha '}, \rho ((xy)^{\alpha '}) f_1(x,y)^{\alpha ''} \xi ^{\alpha ''} \eta ^{\alpha ''})= 
[(x, \xi ) \cdot (y, \eta )]^{\alpha }= (x, \xi )^{\alpha } \ast (y, \eta )^{\alpha }= \nonumber \end{equation}
\begin{equation} 
(x^{\alpha '}, \rho (x^{\alpha '}) \xi ^{\alpha ''}) \ast (y^{\alpha '}, \rho (y^{\alpha '}) \eta ^{\alpha ''})= 
(x^{\alpha '} y^{\alpha '}, f_2(x^{\alpha '}, y^{\alpha '}) \rho (x^{\alpha '}) \xi ^{\alpha ''} \rho (y^{\alpha '}) \eta ^{\alpha ''}) \nonumber \end{equation} 
for all $x,y \in S$ and $ \xi, \eta \in A$. 
It follows $\alpha '$ is an automorphism of the loop $S$, $\alpha ''$ is an automorphism of $A$ and for all $x,y \in S$ identity 
\begin{equation} \label{equ52}  f_1(x,y)^{\alpha ''}= \rho ((xy)^{\alpha '})^{-1} \rho (x^{\alpha '}) \rho (y^{\alpha '}) 
f_2(x^{\alpha '}, y^{\alpha '}) \end{equation} 
holds. This discussion yields the following 

\begin{Prop} \label{isom} Let $(L_1, \cdot )$ and $(L_2, \ast )$  be two loops of type (\ref{equextension}) 
belonging to the functions 
$f_1$ respectively $f_2$. The map $\beta :L_1 \to L_2, (x, \xi ) \mapsto 
(x^{\alpha '}, \rho (x^{\alpha '}) \xi ^{\alpha ''})$ defines an isomorphism of $L_1$ onto $L_2$ mapping the group $A$ onto itself if and only 
if $\alpha '$ is an automorphism of the loop $S$, $\alpha ''$ is an automorphism of $A$, for the map $\rho :S \to A$ 
with $\rho (e)=1 \in A$ one has that the set 
$\{\rho (x^{\alpha '}); \ x \in S, \alpha ' \in Aut(S) \}$ is contained in the centre of $A$ and for all $x, y \in S$   identity (\ref{equ52}) holds. 
\end{Prop} 

\noindent
Let $(S, h)$ be a weighted Steiner loop. Putting $x=y$ into (\ref{equ52}) we obtain 
\begin{equation} \label{equ50}  f_1(x,x)^{\alpha ''} f_2(x^{\alpha '}, x^{\alpha '})^{-1}= 
\rho (x^{\alpha '})^2= 
f_2(x^{\alpha '}, x^{\alpha '})^{-1} f_1(x,x)^{\alpha ''}. \end{equation} 
Since the Steiner loop $S$ is commutative from (\ref{equ52}) we get 
\begin{equation} \label{equsymmetric52}  f_1(y,x)^{\alpha ''}= \rho ((xy)^{\alpha '})^{-1} \rho (x^{\alpha '}) \rho (y^{\alpha '}) 
f_2(y^{\alpha '}, x^{\alpha '}) \nonumber \end{equation} 
for all $x,y \in S$. 
If $x \neq y \in S \setminus \{ e \}$, then using the functions $h_i$, $i=1,2$, from (\ref{equ52}), respectively from the last identity it follows  identity 
\begin{equation}  \label{equ51}  h_1(x)^{\alpha ''} h_1(y)^{\alpha ''}= 
\rho ((x y)^{\alpha '})^{-1} \rho (x^{\alpha '}) \rho (y^{\alpha '}) h_2(x^{\alpha '}) h_2(y^{\alpha '}),   \end{equation}  
respectively 
\begin{equation}  \label{equ53}  h_1(y)^{\alpha ''} h_1(x)^{\alpha ''}= 
\rho ((x y)^{\alpha '})^{-1} \rho (x^{\alpha '}) \rho (y^{\alpha '}) h_2(y^{\alpha '}) h_2(x^{\alpha '}).  \nonumber  \end{equation} 
Comparing these two identities one has 
\begin{equation} \label{equisomruj3}   [h_1(y), h_1(x)]^{\alpha ''} = [h_2(y^{\alpha '}),  h_2(x^{\alpha '})].  
\end{equation}

\noindent
Hence we have the following 

\begin{Coro} \label{isomsteiner} Let $L_1$ and $L_2$ be two Steiner-like loops with respect to the functions $f_1(x,y)$,  
respectively $f_2(x,y)$. The mapping $\beta :L_1 \to L_2; (x, \xi ) \mapsto 
(x^{\alpha '}, \rho (x^{\alpha '}) \xi ^{\alpha ''})$ is an isomorphism of $L_1$ onto $L_2$ mapping the group $A$ onto itself if and only if $\alpha '$ is an automorphism of the Steiner loop $S$, $\alpha ''$ is an automorphism of $A$, the set 
$\{\rho (x^{\alpha '}); \ x \in S, \alpha' \in Aut(S) \}$ is contained in the centre of $A$, 
for all $x \in S$ identities (\ref{equ50}) 
and for all $x, y \in S \setminus \{ e \}$, $x \neq y$ identity (\ref{equ51}) are satisfied.   
Moreover, if $\beta :L_1 \to L_2$ is an isomorphism, then identity (\ref{equisomruj3}) holds. 
\end{Coro} 

\noindent
Now we consider isomorphisms $\beta $ from the Steiner-like loop $L_1$ onto the Steiner-like loop $L_2$ such that the corresponding Steiner loop  
$S$ has more than $2$ elements and such that $\beta $ induces on the factor loop $L_i/A \cong S$ as well as on the group $A$ the identity. 
In this case $\alpha ''= \alpha '=1$ and identities (\ref{equ50}) and (\ref{equ51}) yield that if $L_1$ and $L_2$ are isomorphic, then for all $x,y \in S \setminus \{ e \}$, 
$x \neq y$, we have the following identity: 
\begin{equation} \label{equidentity} h_1(y)^{-1} h_1(x)^{-1} h_1(y)^{-1} h_1(x)^{-1} h_2(x) h_2(y) h_2(x) h_2(y)= \nonumber \end{equation} 
\begin{equation}  f_1(xy, xy) f_2(xy, xy)^{-1} f_1(x,x)^{-1} f_2(x,x) f_1(y,y)^{-1} f_2(y,y). \nonumber \end{equation}

The last identity is automatically satisfied if one of the following properties holds:
\newline
a) both loops are left alternative or have the left inverse property, 
\newline
b) both loops have the cross inverse property, 
\newline
c) both loops have the automorphic inverse property,  
\newline
d) both loops have the weak inverse property and for all $x,y \in S$ identity $h_1(x) h_1(y) h_1(xy)= h_2(x) h_2(y) h_2(xy)=c$ holds 
with a fixed element $c \in A$.

\section{Automorphisms}   
If $(L_1, \cdot )= (L_2, \ast )$ in Section 9 we obtain automorphism of loops of type (\ref{equextension}).  Hence Proposition 
\ref{isom} yields the first part of the following

\begin{Prop} \label{automorphism} Let $L$ be a loop of type (\ref{equextension}).  
A mapping $\alpha : L \to L$ defined by 
$(x, \xi)^{\alpha }=(x^{\alpha '}, \rho(x^{\alpha '}) \xi^{\alpha ''})$ is an automorphism of $L$ if and only if $\alpha '$ is an automorphism of the loop $S$, $\alpha ''$ is an automorphism  of the group $A$, for the map $\rho :S \to A$ with 
$\rho (e)=1$ one has that   
the set $\{\rho (x^{\alpha '}); \ x \in S, \alpha ' \in Aut(S) \}$ is contained  in the centre $Z(A)$ of $A$ and for all $x, y \in S$ identity 
\begin{equation} \label{equ34}   f(x,y)^{\alpha ''}= \rho ((xy)^{\alpha '})^{-1} \rho (x^{ \alpha '}) 
\rho (y^{ \alpha '}) f(x^{\alpha '}, y^{\alpha '}) \end{equation} 
holds. The map $\rho : S \to Z(A)$, $x \mapsto \rho (x)$ is a homomorphism  from $S$ into $Z(A)$ if and only if there are 
$\alpha ' \in Aut(S)$, $\alpha '' \in Aut(A)$ such that 
$f(x,y)^{\alpha ''}=f(x^{\alpha '}, y^{\alpha '})$ for all $x,y \in S$. 
\end{Prop}
\begin{proof}  The last assertion is a direct  consequence of identity (\ref{equ34}). 
\end{proof}

\noindent
The kernel of the homomorphism $\rho $ contains the derived subloop $S'$ of the loop $S$ since $S'$ is the smallest subloop of $S$ such that $S/S'$ is
abelian. 

Conversely, any  homomorphism $\rho $ from $S$ into the centre $Z(A)$ determines an automorphism of $L$ which is the mapping $\beta _{\rho }: (x, \xi ) \mapsto (x, \rho (x) \xi )$. The set of these automorphisms forms the normal subgroup 
$\Psi $ of the automorphism group $\Gamma $ of $L$ which consists of automorphisms inducing on $S$ as well as on  $A$ the identity. The group $\Psi $ is commutative. 

Moreover, let $\alpha '$ be an automorphism of $S$ such that $f(x^{\alpha '}, y^{\alpha '})=  f(x,y)$ for all $x,y \in S$. Then 
the elements $\beta _{\alpha '}: (x, \xi ) \to (x^{\alpha '}, \xi)$ form a subgroup $\Sigma _1$ of the automorphism group $\Gamma $ of $L$. 
Let  $\alpha ''$ be an automorphism of $A$ such that $f(x, y)=  f(x,y)^{\alpha ''}$ for all $x,y \in S$. Then the elements 
$\beta _{ \alpha ''}: (x, \xi ) \mapsto (x, \xi ^{\alpha ''})$ form a subgroup $\Sigma _2$ of the automorphism group $\Gamma $ of $L$.

Let $\alpha '$ be an automorphism of $S$ and let $\alpha ''$ be an automorphism of $A$ such that 
$f(x^{\alpha '}, y^{\alpha '})=  f(x,y)^{\alpha ''}$ for all $x,y \in S$. Any such pair 
$(\alpha ', \alpha '')$ determines an automorphism of $L$, namely $\beta _{\alpha ', \alpha ''}= \beta _{\alpha '} \beta _{\alpha ''}= \beta _{\alpha ''} \beta _{\alpha '}: (x, \xi ) \mapsto (x^{\alpha '}, \xi ^{\alpha ''})$. The set of these automorphisms forms a group $\Sigma $ such that $\Psi \cap \Sigma =\{ 1 \}$.   

If $(L_1, \cdot )=(L_2, \ast )$ Corollary \ref{isomsteiner} yields the following

\noindent
\begin{Coro} \label{steinerautomorphismuj} Let $L$ be a  Steiner-like loop. The map $\beta : L \to L$ defined by 
$(x, \xi)^{\beta }=(x^{\alpha '}, \rho(x^{ \alpha '}) \xi^{\alpha ''})$ is an automorphism  of $L$ if and only if  for all $x \in S$  identity 
\begin{equation} \label{equsteineruj1}     
\rho(x^{\alpha '})^2 f(x^{\alpha '}, x^{\alpha '}) = f(x,x)^{\alpha ''} \end{equation}
and for all $x,y \in S \setminus \{ e \}$, $x \neq y$ identity 
\begin{equation} \label{equsteineruj2}  \rho ((xy)^{\alpha '})^{-1} \rho(x^{\alpha '}) \rho(y^{\alpha '}) h(x^{\alpha '}) 
h(y^{\alpha '})=  h(x)^{\alpha ''} h(y)^{\alpha ''}  \end{equation} 
hold. 
Moreover, if $\beta $ is an automorphism of $L$, then one has  
\begin{equation} \label{equsteineruj3}   [h(x), h(y)]^{\alpha ''} = [h(x^{\alpha '}),  h(y^{\alpha '})].  \nonumber 
\end{equation} 
\end{Coro}

\begin{Coro} \label{steinerautomorphism} Let $L$ be a  Steiner-like loop. 
\newline
(a) The normal subgroup $\Psi $ of the automorphism group $\Gamma $ of $L$ is an elementary abelian $2$-group. 
\newline
(b) Assume that $\rho : S \to Z(A)$ is a homomorphism and $S$ has more than 2 elements. Then the map $\beta : L \to L$ 
defined by 
$(x, \xi)^{\beta }=(x^{\alpha '}, \rho(x^{ \alpha '}) \xi^{\alpha ''})$ is an automorphism  of $L$ 
if and only if  for all $x \in S \setminus \{ e \}$ identities 
\begin{equation} \label{equsteiner1} 
 f(x,x)^{\alpha ''} =  f(x^{\alpha '}, x^{\alpha '}) \ \ \hbox{and} \ \ c \ h(x^{\alpha '}) = h(x^{\alpha '}) c^{-1}= h(x)^{\alpha ''} \nonumber \end{equation} 
are satisfied, where $c$ is a suitable element of $A$.  
\end{Coro} 
\begin{proof} Since for all $x \in S \setminus \{ e \}$ one has $o(x)=2$ and $\rho :S \to Z(A)$ is a homomorphism any element $\rho (x)$ has order at most $2$. Therefore $id \neq \beta _{\rho } \in \Psi $ is an involution and (a) holds.  

As for all $x \in S$ the element $\rho (x)$ has order at most $2$ identity (\ref{equsteineruj1})  yields the first identity of the Corollary and 
identity (\ref{equsteineruj2}) reduces to 
\begin{equation} \label{equh1} [h(x)^{\alpha ''}]^{-1} h(x^{\alpha '}) = h(y)^{\alpha ''} [h(y^{\alpha '})]^{-1}=c   \end{equation} 
for all $y \neq x \in S \setminus \{ e \}$. Let $z $ be an element of $S \setminus \{ e, x, y \}$. Then replacing $y$ by $z$,   respectively  $x$ by $z$    in (\ref{equh1}) we obtain  $h(z)^{\alpha ''}= c h(z^{\alpha '})$, respectively 
$h(z^{\alpha '})= h(z)^{\alpha ''} c$. Comparing the last equations we get  
$h(z^{\alpha '})=c h(z^{\alpha '}) c$ for all $z \in S \setminus \{ e \}$.

Conversely, let $f(x,x)^{\alpha ''} =f(x^{\alpha '}, x^{\alpha '})$ and 
$$[h(x)^{\alpha ''}]^{-1} h(x^{\alpha '})= h(x)^{\alpha ''}  h(x^{\alpha '})^{-1}=c$$ for all $x \in S \setminus \{ e \}$. Since $\rho $ is a homomorphism $\rho (S)$ is an elementary abelian $2$-group. Hence the identity $f(x,x)^{\alpha ''} =  f(x^{\alpha '}, x^{\alpha '})$ in the Corollary yields identity (\ref{equsteineruj1}) and identity (\ref{equsteineruj2}) reduces to $h(x)^{\alpha ''} h(y)^{\alpha ''}= h (x^{\alpha '}) 
h(y^{\alpha '})$. Putting in this identity $h(x)^{\alpha ''} =c h(x^{\alpha '})$ we get $c h(x^{\alpha '}) c= 
h (x^{\alpha '})$  which is true because of (\ref{equh1}). \end{proof}

Author's address: \'Agota Figula \\
Institute of Mathematics, University of Debrecen \\
H-4010 Debrecen, P.O.B. 12 \\
Hungary, figula@math.klte.hu 

\medskip
\noindent 
Karl Strambach \\
Department Mathematik, Universit\"at Erlangen-N\"urnberg \\ 
D-91054 Erlangen, Bismarkstra\ss e 1 1/2 \\ 
Germany, strambach@mi.uni-erlangen.de

\end{document}